\documentclass[11pt]{amsart}
\usepackage{amsfonts}
\usepackage{amsmath}
\usepackage{amsthm}
\usepackage{bm}
\usepackage[dvips]{graphicx}
\usepackage{natbib}
\usepackage{tikz}
\usepackage{bigints}

\numberwithin{equation}{section}

\allowdisplaybreaks

\theoremstyle{plain}
\newtheorem{thm}{Theorem}[section]

\newtheorem{lem}[thm]{Lemma}

\newtheorem{prop}[thm]{Proposition}
\newtheorem{proposition}[thm]{Proposition}
\theoremstyle{definition}

\newtheorem{remark}[thm]{Remark}

\usepackage{amsfonts}
\usepackage{amsmath}
\usepackage{amsthm}
\usepackage{bm}
\usepackage{natbib}
\usepackage{bigints}

\numberwithin{equation}{section}

\allowdisplaybreaks

\def\beqn{\begin{equation}}
\def\beqn*{$$}
\def\eeqn{\end{equation}}
\def\eeqn*{$$}

\def\nn{\nonumber}
\newcommand{\BX}{{\bf X}}
\newcommand{\BY}{{\bf Y}}
\newcommand{\bx}{{\bf x}}

\newcommand{\bi}{{\bf i}}

\newcommand{\id}{infinitely divisible}
\newcommand{\reals}{{\mathbb R}}
\newcommand{\bbr}{\reals}
\newcommand{\bbn}{{\mathbb N}}
\newcommand{\vep}{\varepsilon}
\newcommand{\bbz}{\protect{\mathbb Z}}

\newcommand{\SaS}{S$\alpha$S}
\newcommand{\eid}{\stackrel{d}{=}}

\newcommand{\one}{{\bf 1}}
\newcommand{\rhoinv}{\rho^{\leftarrow}}

\def\sas{{S$\alpha$S}}
\def\sgs{{S$\gamma$S}}

\def\bi{\begin{itemize}}
\def\ei{\end{itemize}}

\def\bb{\beta}

\def\eps{\epsilon}

\def\be{\begin{equation}}
\def\ee{\end{equation}}
\def\bea{\begin{eqnarray}}
\def\eea{\end{eqnarray}}
\def\nn{\nonumber}
\def\ff{\infty}
\def\({\left(}
\def\){\right)}
\def\[{\left[}
\def\]{\right]}
\def\lk{\left\{}
\def\rk{\right\}}
\def\lb{\left|}
\def\rb{\right|}

\def\Cov{\text{Cov}}

\def\X{\mathbb{X}}
\def\Y{\mathbb{Y}}
\def\N{\mathbb{N}}
\def\R{\mathbb{R}}

\def \P{\mathbf{P}}
\def \Q{\mathbf{Q}}
\def \E{\mathbf{E}}
\def \Y{Y_{\alpha,\bb,\gamma}}
\def \M{M_{\beta}}
\def \MM{M_{\beta}((t-x)_{+}}

\def \srlim{\stackrel{n\to\infty}}

\def \aaa{\mathfrak	{a}}

\def \CC{{\mathcal C}}
\def \DD{{\mathcal D}}
\def \EE{{\mathcal E}}
\def \FF{{\mathcal F}}

\def \TT{{\mathcal T}}

\def \XX{{\mathcal X}}

\begin{document}

\title[Functional Central Limit Theorem]{Functional central limit
  theorem for negatively dependent heavy-tailed stationary infinitely
  divisible processes generated by conservative flows}
\author{Paul Jung}
\address{Department of Mathematics \\
University of Alabama \\
Birmingham, AL 35294}
\email{pjung@uab.edu}

\author{Takashi Owada}
\address{Faculty of Electrical Engineering\\
Technion - Israel Institute of Technology  \\
Haifa, 32000, Israel}
\email{takashiowada@ee.technion.ac.il}

\author{Gennady Samorodnitsky}
\address{School of Operations Research and Information Engineering\\
and Department of Statistical Science \\
Cornell University \\
Ithaca, NY 14853}
\email{gs18@cornell.edu}

\thanks{Jung's research was partially supported by NSA grant H98230-14-1-0144. Owada's research was partially supported by URSAT, ERC Advanced Grant 320422. Samorodnitsky's research was partially supported by the ARO
grant  W911NF-12-10385 at Cornell University.}

\keywords{infinitely divisible process, conservative flow, Harris
  recurrent Markov chain,  functional central
  limit theorem, self-similar process, pointwise dual ergodicity,
  Darling-Kac theorem, fractional stable motion
\vspace{.5ex}}

\begin{abstract}
We prove a functional central limit theorem for
partial sums of symmetric stationary long range dependent heavy tailed
infinitely divisible  processes with a certain type of negative
dependence. Previously only positive dependence could be treated. The
negative dependence involves cancellations of the Gaussian second
order. This 
leads to new types of  {limiting} processes involving stable random
measures, due to heavy tails, Mittag-Leffler processes, due to long
memory, and Brownian motions, due to the Gaussian second order
cancellations.  
\end{abstract} 
 
\maketitle

\section{Introduction} \label{sec:intro}
 
Let $\BX=(X_1,X_2,\ldots)$ be a discrete time stationary stochastic
process; depending on notational convenience we will sometimes allow
the time index to extend to the entire $\bbz$. 
Assume that $\BX$ is symmetric (i.e. that $\BX\eid -\BX$) and
that the marginal law of $X_1$ is in the domain of attraction of an
$\alpha$-stable law, $0<\alpha<2$. That is, 
\begin{equation} \label{e:marginal.tail}
P\bigl( |X_1|>\cdot\bigr) \in RV_{-\alpha} \ \ \text{at infinity;}
\end{equation}
see \cite{feller:1971} or \cite{resnick:1987}. Here and elsewhere in
this paper we use the notation $RV_p$ for the set of functions of
regular variation with exponent $p\in\bbr$. If the process satisfies a
functional central limit theorem, then a statement of the type
\begin{equation} \label{e:FCLT} 
\left( \frac{1}{c_n}\sum_{k=1}^{\lfloor nt\rfloor} X_k,  \, 0\leq
  t\leq 1\right) \Rightarrow \Bigl( Y(t), \, 0\leq   t\leq 1\Bigr)
\end{equation}
holds, with $(c_n)$  a positive sequence growing to infinity, and
$\BY= \Bigl( Y(t), \, 0\leq   t\leq  1\Bigr)$  a
non-degenerate (non-deterministic) process. The convergence is either
weak convergence in the appropriate topology on $D[0,1]$ or just
convergence in finite dimensional distributions. The heavy tails in
\eqref{e:marginal.tail} will necessarily affect the order of magnitude
of the normalizing sequence $(c_n)$ and the nature of the limiting
process $\BY$. The latter process is, under mild assumptions,
self-similar, with stationary increments; see \cite{lamperti:1962} and
\cite{embrechts:maejima:2002}. If the process $\BX$ is long range
dependent, then both the sequence $(c_n)$ and the  limiting
process $\BY$ may be affected by the length of the memory as well. 

A new class of central limit theorems for long range dependent
stationary processes with heavy tails was introduced in
\cite{owada:samorodnitsky:2015}. In that paper the process $\BX$
was a stationary \id\ process given in the form 
\begin{equation} \label{e:the.process}
X_n = \int_E  f\circ T^n(s) \, dM(s), \quad n=1,2,\dots\,,
\end{equation}
where $M$ is a symmetric homogeneous infinitely divisible random measure on a
measurable space $(E,\mathcal{E})$,  without a Gaussian component,
with control measure $\mu$, 
$f:E \to \mathbb{R}$ is a measurable function, and $T:E \to E$ a
measurable map, preserving the measure $\mu$; precise definitions of
these and following notions are below. The regularly varying tails, in
the sense of \eqref{e:marginal.tail}, of the process
$\BX$ are due to the random measure $M$, 
while the long memory is due
to the ergodic-theoretical properties of the map $T$, assumed to be
conservative and ergodic. In the model considered in
\cite{owada:samorodnitsky:2015} the length of the memory could be
quantified by a single parameter  {$0\le\beta\le 1$ (the larger $\beta$ is},
the longer the memory). Under the crucial assumption that
\begin{equation} \label{e:nonzero.muf}
\mu(f):= \int_E f(s)\, \mu(ds) \not= 0
\end{equation}
(with the integral being well defined), it turns out that the
normalizing sequence  $(c_n)$ is regularly varying with exponent
$H=\beta+(1-\beta)/\alpha$, and the limiting process $\BY$ is, up to a
multiplicative factor of $\mu(f)$, the $\beta$-Mittag-Leffler
fractional symmetric $\alpha$-stable (\SaS) motion defined by 
\begin{equation} \label{eq:MLSM} 
Y_{\alpha,\beta}(t) = \int_{\Omega^{\prime} \times [0,\infty)}
M_{\beta}\bigl((t-s)_+,\omega^{\prime}\bigr) d
Z_{\alpha,\beta}(\omega^{\prime},s), \quad t \geq 0, 
\end{equation}
where $Z_{\alpha,\beta}$ is a \SaS\ random measure on $\Omega^{\prime}
\times [0,\infty)$ with control measure $\P^{\prime} \times \nu_\beta$. Here 
$\nu_\beta$  {is} a measure on $[0,\infty)$ given by 
$\nu_\beta(dx) = (1-\beta) x^{-\beta} \, dx, \, x>0$, and $M_\beta$ is a
Mittag-Leffler process defined on a probability space 
$(\Omega^{\prime},\mathcal{F}^{\prime},\P^{\prime})$ (all the notions
will be defined momentarily). The random measure
$Z_{\alpha,\beta}$ and the process
$Y_{\alpha,\beta}$, are defined on some  probability space 
$(\Omega,\mathcal{F} ,\P )$.

The $\beta$-Mittag-Leffler fractional  \SaS\ motion is a self-similar
process with Hurst exponent $H$ as above. Note that
$$
H\in \left\{ \begin{array}{ll} 
(1, {1/\alpha]} & \text{if $0<\alpha<1$,} \\
\{ 1\} & \text{if $\alpha=1$,} \\
(1/\alpha,1) & \text{if $1<\alpha<2$,}
\end{array}
\right.
$$
which is the top part of the feasible region 
$$
H\in \left\{ \begin{array}{ll} 
(0,1/\alpha] & \text{if $0<\alpha<1$,} \\
(0,1] & \text{if $\alpha=1$,} \\
(0,1) & \text{if $1<\alpha<2$}
\end{array}
\right.
$$
for the Hurst exponent of a self-similar \SaS\ process with stationary
increments; see \cite{samorodnitsky:taqqu:1994}. This is usually
associated with positive dependence both in the increments of the
process $\BY$ itself and the original process $\BX$ in the functional
central limit theorem \eqref{e:FCLT}; the best-known example is that
of the Fractional Brownian motion, the Gaussian self-similar  process
with stationary increments. For the latter process the range of $H$ is the
interval $(0,1)$, and positive dependence corresponds to the range
$H\in (1/2,1)$. 

In the Gaussian case of the Fractional Brownian motion, negative
dependence ($0<H<1/2$) is often related to ``cancellations'' between
the observations; the statement
$$
\sum_{n=-\infty}^\infty \Cov(X_0,X_n)=0
$$
is trivially true if the process $\BX$ is the increment process of the
Fractional Brownian motion with $H<1/2$, and the same is true in most
of the situations in \eqref{e:FCLT}, when the limit process is  the
Fractional Brownian motion with $H<1/2$. 

In the infinite variance case considered in
\cite{owada:samorodnitsky:2015}, ``cancellations'' appear when the
integral $\mu(f)$ in \eqref{e:nonzero.muf} vanishes. It is the purpose
of the present paper to take a first step towards understanding this case,
when the long memory due to the map $T$ interacts with the negative
dependence due to the cancellations. We use the cautious formulation
above because with the integral $\mu(f)$ vanishing, the second order
behaviour of $f$ becomes crucial, and in this paper we only consider a
Gaussian type of  {second} order behaviour. Furthermore, even in this
case our assumptions on the space $E$ and map $T$ in
\eqref{e:the.process} are more restrictive than those in
\cite{owada:samorodnitsky:2015}. Nonetheless, we still obtain an
entirely new class of functional limit theorems and limiting
fractional \SaS\ motions. 

This paper is organized as follows. In Section \ref{sec:setup} we
provide the necessary background on infinitely divisible and stable
processes and integrals, and related notions, used in this paper. In
Section \ref{sec:limits} we describe a new class of self-similar \SaS\
processes with stationary   increments, some of which will appear as
limits in the functional central limit theorem proved later. Certain
facts on general state space Markov chains, needed to define and treat the
model considered in the paper, are in Section \ref{sec:markov}. The
main result of the paper is stated and proved in Section
\ref{sec:FCLT}. Finally, Section \ref{sec:moments} is an appendix
containing bounds on fractional moments of \id\ random variables
needed elsewhere in the  {paper.} 

We will use several common abbreviations  {throughout} the paper: {\it ss}
for ``self-similar'', {\it sssi} for ``self-similar, with stationary
increments'', and {\it \SaS\ } for ``symmetric $\alpha$-stable''. 

\section{Background} \label{sec:setup} 

In this paper we will work with symmetric infinitely divisible processes defined
as integrals of deterministic functions with respect to homogeneous
symmetric \id\ random
measures, the symmetric stable processes and measures forming a special case. 
Let $(E,\mathcal{E})$ be a measurable space. Let $\mu$ be  a
$\sigma$-finite measure on $E$, it will be assumed to be infinite in
 {most} of the paper, but at the moment it is not important. Let $\rho$
be a one-dimensional symmetric L\'evy measure, i.e. a $\sigma$-finite
measure on $\bbr\setminus \{0\}$ such that 
$$
\int_\bbr \min(1,x^2)\, \rho(dx)<\infty\,.
$$
If $\mathcal{E}_0=\bigl\{ A\in \EE:\, \mu(A)<\infty\}$, then a homogeneous
symmetric \id\ random measure $M$ on $(E,\mathcal{E})$ with control
measure $\mu$ and local  L\'evy measure $\rho$ is a stochastic process 
$\bigl( M(A), \, A\in\EE_0\bigr)$ such that 
\begin{equation*}  
 {\E} e^{iu M(A)} = \exp\left\{ -\mu(A) \int_{\mathbb{R}}
  \bigl(1-\cos(ux)\bigr)\, \rho(dx) \right\}\,  \quad u \in \bbr
\end{equation*}
for every $A \in\EE_0$. The random measure $M$ is independently
scattered and $\sigma$-additive in the usual sense of random measures;
see \cite{rajput:rosinski:1989}. The random measure is symmetric
$\alpha$-stable (\SaS), $0<\alpha<2$, if 
$$
\rho(dx) = \alpha |x|^{-(\alpha+1)}\, dx\,.
$$

If $M$ has a control measure $\mu$ and a local  L\'evy measure $\rho$,
and $g:\, E\to\bbr$ is a measurable function such that 
\begin{equation} \label{e:integrability}
\int_E \int_\bbr \min\bigl( 1, x^2g(s)^2\bigr)\, \rho(dx)\,
\mu(ds)<\infty\,,
\end{equation}
then the integral $\int_E g\, dM$ is well defined and is a symmetric
\id\ random variable. In the $\alpha$-stable case the integral is a
\SaS\ random variable and the integrability condition \eqref{e:integrability}
reduces to the $L^\alpha$ condition
\begin{equation} \label{e:integrability.stable}
\int_E |g(s)|^\alpha\, \mu(ds)<\infty\,.
\end{equation}
We remark that in the $\alpha$-stable case it is common to use the
$\alpha$-stable version of the control measure; it is just a scaled
version $C_\alpha\, \mu$ of the control measure $\mu$, with $C_\alpha$
being the $\alpha$-stable  tail constant given by 
$$
C_{\alpha} = \left( \int_0^{\infty} x^{-\alpha} \sin x \,  dx
\right)^{-1} = \begin{cases} (1-\alpha) / \bigl(\Gamma(2-\alpha) \cos(\pi
  \alpha / 2)\bigr) & \ \ \text{if } \alpha \neq 1, \\ 
2 / \pi  & \ \ \text{if } \alpha=1\,.
\end{cases}
$$
See  \cite{rajput:rosinski:1989}  for this and the subsequent
properties of \id\ processes and integrals. 

We will consider symmetric \id\ stochastic processes (without a
Gaussian component) $\BX$ given in the form
$$
X(t) = \int_E g(t, s)\, M(ds)\,, \ t\in \TT\,,
$$
where $\TT$ is a parameter space, and $g(t,\cdot)$ is, for each $t\in
\TT$, a measurable function satisfying \eqref{e:integrability}. The
(function level) L\'evy measure of the process $\BX$ is given by 
\begin{equation} \label{e:measure.process}
\kappa_\BX = (\rho \times \mu) \circ K^{-1}\,,
\end{equation}
with $K:\, \bbr \times E \to \bbr^\TT$ given by $K(x,s) = x\bigl(
g(t,s), \, t\in \TT\bigr)$, $s\in E,\, x\in \bbr$. 

An important special case  {for us} is that of  $\TT= \bbn$ and 
\begin{equation} \label{e:shifted.f}
g(n,s) = f\circ T^n(s), \ n=1,2,\dots\,,
\end{equation} 
where $f:\, E\to\bbr$ is a measurable function satisfying
\eqref{e:integrability}, and $T:E \to E$ a
measurable map, preserving the control measure $\mu$. In this  {case} we
obtain the process exhibited in \eqref{e:the.process}. It is
elementary to check that in this case the L\'evy measure $\kappa_\BX$
in \eqref{e:measure.process} is invariant under the left shift
$\theta$ on $\bbr^\bbn$, 
$$
\theta (x_1,x_2,x_3,\ldots) = (x_2,x_3,\ldots)\,.
$$
In particular, the process $\BX$ is, automatically, stationary.  There
is a close relation between certain ergodic-theoretical properties of
the shift operator $\theta$ with respect to the L\'evy measure
$\kappa_\BX$ (or of the map $T$ with respect to the control measure
$\mu$) and certain distributional properties of the stationary process
$\BX$; we will discuss these below.

Switching gears a bit, we now recall a crucial notion needed for the
main result of this paper as well as for the presentation of the new
class of fractional \SaS\ noises in the next section. 
For $0<\beta<1$, let $\( S_{\beta}(t)\)$ be a
$\beta$-stable subordinator, a L\'evy process with increasing
sample paths, satisfying $\E e^{-\theta S_{\beta}(t)} = \exp \{ -t
\theta^{\beta} \}$ for $\theta\geq 0$ and $t\geq 0$.
The {\it Mittag-Leffler process} is its inverse process given by
\begin{equation} \label{MLprocess}
M_{\beta}(t) := S_{\beta}^{\leftarrow}(t) =
 \inf \bigl\{u\geq 0:\  S_{\beta}(u) \geq t  \bigr\}, \
t\geq 0\,.
\end{equation}
It is a continuous process with nondecreasing sample paths. Its 
 marginal distributions are the Mittag-Leffler distributions, whose 
Laplace transform is finite for all real values of the argument and is
given by 
\begin{equation} \label{MT.transform}
\E  \exp \{ \theta M_{\beta}(t) \} = \sum_{n=0}^{\infty} \frac{
  (\theta t^{\beta})^n}{\Gamma(1+n \beta)}, \quad \theta \in
\mathbb{R}; 
\end{equation}
see Proposition 1(a) in \cite{bingham:1971}. Using
\eqref{MT.transform},  the definition of the Mittag-Leffler
process can be naturally extended to the boundary cases $\beta=0$ and
$\beta=1$. We do this by setting $M_0(0)=0$ and $M_0(t)= {E_{\rm st}}$, $t>0$, with $ {E_{\rm st}}$ a
standard exponential random variable, and $M_1(t)=t$, $t\geq 0$. 

The Mittag-Leffler process is self-similar with exponent $\beta$. 
It has neither stationary  nor independent increments (apart from the
degenerate case $\beta=1$). 

\section{A new class of self-similar \SaS\ processes with stationary
  increments} \label{sec:limits}

In this section we introduce a new class of self-similar \SaS\ processes with stationary
  increments. A subclass of these processes will appear as a weak
  limit in the functional central limit theorem in Section \ref{sec:FCLT}, but
  the entire class has intrinsic interest. Furthermore, we  {anticipate}
  that other members of the class will appear in other limit
  theorems. The processes in this class are defined, up to a scale
  factor, by 3 parameters, $\alpha, \beta$ and $\gamma$:
\begin{equation*} \label{e:parameters}
0<\alpha<\gamma\leq 2, \ 0\leq \beta\leq 1\,.
\end{equation*}

We proceed with a setup similar to the one in \eqref{eq:MLSM}. Define
a $\sigma$-finite measure on $[0,\infty)$ by 
\begin{equation*} \label{e:nu.beta}
\nu_\beta(dx) = \left\{ \begin{array}{ll}
(1-\beta)x^{-\beta}\, dx, & 0\leq \beta<1,\\
\delta_0 (dx), & \beta=1
\end{array}
\right.
\end{equation*}
($\delta_0$ being the point mass at zero). Let $(\Omega',\FF',\P')$ be
a probability space, and let $\(S_\gamma(t,\omega')\)$ be a \sgs\ L\'evy motion and
$\(M_\beta(t,\omega')\)$ be an independent $\beta$-Mittag-Leffler
process, both defined on $(\Omega',\FF',\P')$. We define
\begin{equation}\label{def:Y process}
\Y(t):=\int_{\Omega'\times[0,\infty)} S_\gamma(\MM,\omega'),\omega')\, 
dZ_{\alpha, \beta}(\omega',x), \ t\ge 0\,,
\end{equation}
where $Z_{\alpha,\beta}$ is a \sas\ random measure on $\Omega'\times[0,\infty)$
with control measure $\P'\times\nu_\beta$; we use the $\alpha$-stable
version of the control measure  {(see the remark following \eqref{e:integrability.stable})}. 

\begin{remark} \label{rk:boundary}
The boundary cases $\beta=0$ and $\beta=1$ are somewhat special. In
the case $\beta=0$ we interpret the process in \eqref{def:Y process}  as
$$
Y_{\alpha,0,\gamma}(t)=\int_{\Omega'\times[0,\infty)}
S_\gamma( {E_{\rm st}}(\omega'),\omega') \one(x<t)\, 
dZ_{\alpha,0}(\omega',x), \ t\ge 0\,,
$$
where  {$E_{\rm st}$} is a standard exponential random variable defined on
$(\Omega',\FF',\P')$, independent of $\(S_\gamma(t,\omega')\)$, while 
$Z_{\alpha,0}$ is a \sas\ random measure on $\Omega'\times[0,\infty)$
with control measure $\P'\times {\rm Leb}$. It is elementary to see
that this process is a \sas\ L\'evy motion itself, and the dependence on
$\gamma$ is only through a multiplicative constant, equal to
$$
\bigl( \E^\prime |S_\gamma(1)|^\alpha \E^\prime
( {E^{\alpha/\gamma}_{\rm st}})\bigr)^{1/\alpha}\,.
$$

In the second boundary case $\beta=1$  the variable $x$ in the
integral becomes redundant, and we interpret  the process in
\eqref{def:Y process}  as 
$$
Y_{\alpha,1,\gamma}(t)=\int_{\Omega'}
S_\gamma(t,\omega') \, 
dZ_{\alpha,1}(\omega'), \ t\ge 0\,,
$$
where $Z_{\alpha,1}$ is a \sas\ random measure on $\Omega'$
with control measure $\P'$. This process is, distributionally, 
sub-stable, with an alternative representation   
$$
Y_{\alpha,1,\gamma}(t)=W^{1/\gamma} S_\gamma(t), \  t\geq 0\,,
$$
with $W$ a positive strictly $\alpha/\gamma$-stable random variable
independent of the \sgs\ L\'evy motion $(S_\gamma)$, both of which are
now defined on 
$(\Omega,\FF,\P)$; see Section 3.8 in \cite{samorodnitsky:taqqu:1994}.  
\end{remark}

\begin{prop}\label{pr:new.pr}
 $(\Y(t))$ is a well defined $H$-sssi \sas\ process with
\begin{equation}\label{eq:H}
H=\frac \beta \gamma+ \frac {1-\beta}{\alpha}.
\end{equation}
\end{prop}
\begin{proof}
The boundary cases $\beta=0$ and $\beta=1$ are discussed in Remark
\ref{rk:boundary}, so we will consider now the case $0<\beta<1$.

The argument is similar to that of Theorem 3.1 in
\cite{owada:samorodnitsky:2015}.  
To see that $(\Y(t))$ is well defined, notice that for $t>0$, by
self-similarity of $(S_\gamma)$, 
\begin{eqnarray*}
&&\E'\int_0^{\infty}|S_\gamma(\MM))|^\alpha(1-\beta)x^{-\beta}dx \nn \\
&=& \E'|S_\gamma(1)|^\alpha\, 
\E'\int_0^{\infty}\MM)^{\alpha/\gamma}(1-\beta)x^{-\beta}dx\\
&\le& \E'|S_\gamma(1)|^\alpha\, \E'\M(t)^{\alpha/\gamma} t^{1-\bb}
      <\infty\,; \nn
\end{eqnarray*}
the finiteness of $\E'|S_\gamma(1)|^\alpha$ follows since $\alpha<\gamma$.
Next, let $c>0$, $t_1,\ldots,t_k>0$, and
$\theta_1,\ldots,\theta_k\in\R$.  We have
\begin{eqnarray*}
\E'\exp\lk i\sum_{j=1}^k\theta_j\Y(ct_j)\rk&=& \exp\lk -\int_0^\infty
\E'\lb\sum_{j=1}^k\theta_jS_\gamma(\M((ct_j-x)_+))\rb^\alpha (1-\bb)x^{-\bb} dx \rk\\
&=&\nn\exp\lk-c^{1-\beta}\int_0^\infty
\E'\lb \sum_{j=1}^k \theta_j S_\gamma(\M(c(t_j-y)_+)\rb^\alpha (1-\beta) y^{-\beta} dy\rk\,,
\end{eqnarray*}
where we substituted $x=cy$ in the last step.  Because of the
self-similarity of $(\M)$ and $(S_\gamma)$, the above equals
\begin{eqnarray}
&&\nn \exp\lk -c^{1-\beta}\int_0^\infty
\E'\lb \sum_{j=1}^k \theta_j c^{\beta/\gamma}S_\gamma(\M((t_j-y)_+))\rb^\alpha (1-\beta) y^{-\beta} dy\rk\\
&=& \nn \exp\lk-c^{\alpha H}\int_0^\infty
\E'\lb \sum_{j=1}^k \theta_j S_\gamma(\M((t_j-y)_+))\rb^\alpha (1-\beta) y^{-\beta} dy\rk\\
&=& \nn \exp\lk i\sum_{j=1}^k\theta_jc^H\Y(t_j)\rk.
\end{eqnarray}
This shows that $\Y$ is $H$-ss with $H$ given by \eqref{eq:H}.

Finally, we check stationarity of the increments of $\Y$.  We must
check that for any $s>0$, $t_1,\ldots,t_k>0$, and
$\theta_1,\ldots,\theta_k\in\R$,
\begin{eqnarray*}
&&\nn\int_0^\infty
\E'\lb \sum_{j=1}^k\theta_j\bigl[ S_\gamma(\M((t_j+s-x)_+) )- S_\gamma(M_\bb((s-x)_+))\bigr]\rb^\alpha x^{-\bb} dx \\
&=& \int_0^\infty
\E'\lb \sum_{j=1}^k\theta_j S_\gamma(M_\bb((t_j-x)_+))\rb^\alpha x^{-\bb} dx.
\end{eqnarray*}
Split the integral in the left-hand side according to the sign of $s-x$ and use
the substitutions $r=s-x$ and $-r=s-x$  to get
\begin{eqnarray}\label{eq:SI split}
&&\nn\int_0^s
\E'\lb\sum_{j=1}^k\theta_j\bigl[ S_\gamma(M_\bb(t_j+r)) - S_\gamma(M_\bb(r))\bigr]\rb^\alpha (s-r)^{-\bb} dr \\
&+&\nn \int_0^\infty \E'\lb\sum_{j=1}^k\theta_j S_\gamma(M_\bb((t_j-r)_+))\rb^\alpha (s+r)^{-\bb} dr.
\end{eqnarray}
Rearranging terms, we are left to check
\begin{eqnarray}\label{eq:SI}
&&\nn\int_0^s
\E'\lb\sum_{j=1}^k\theta_j\bigl[ S_\gamma(M_\bb(t_j+r)) - S_\gamma(M_\bb(r))\bigr]\rb^\alpha (s-r)^{-\bb} dr \\
&=& \int_0^\infty \E'\lb\sum_{j=1}^k\theta_j
S_\gamma(M_\bb((t_j-x)_+))\rb^\alpha \( x^{-\bb}-(s+x)^{-\bb}\) dx.
\end{eqnarray}
However, by the stationarity of the increments of $(S_\gamma)$, the
left-hand side of \eqref{eq:SI} reduces to
$$
\int_0^s
\E'\lb\sum_{j=1}^k\theta_j S_\gamma\bigl( M_\bb(t_j+r)-M_\bb(r)\bigr)\rb^\alpha
(s-r)^{-\bb} dr.
$$

Let $\delta_r = S_\beta\( M_\beta(r)\)-r$ be the overshoot of
the level $r>0$ by the $\beta$-stable subordinator $\(
S_{\beta}(t)\)$ related to $\(
M_{\beta}(t)\)$ by \eqref{MLprocess}. The law of
$\delta_r$ is given by
$$
P(\delta_r \in dx) = \frac{\sin \beta \pi}{\pi} r^{\beta} (r+x)^{-1}
x^{-\beta}\, dx, \ x>0\,;
$$
see  Exercise 5.6 in \cite{kyprianou:2006}. Further, by the strong
Markov property of the stable subordinator we have
$$
\( M_{\beta}(t+r) - M_{\beta}(r), \, t\geq 0 \)
\stackrel{d}{=} \( M_{\beta}((t-\delta_r)_+), \, t\geq 0 \)\,,
$$
with the understanding that $M_\beta$ and $\delta_r$ in the right-hand
side  are independent. We conclude that
\begin{equation} \label{e:si.check.2}
\int_{0}^s \E' \left|\sum_{j=1}^k \theta_j S_\gamma\bigl( M_{\beta}(t_j+r) -
M_{\beta}(r)\bigr)\right|^{\alpha} (s-r)^{-\beta} dr
\end{equation}
$$
= \frac{\sin \beta \pi}{\pi} \int_0^{\infty} \int_0^s \E'
\left|\sum_{j=1}^k \theta_j S_\gamma \bigl( M_{\beta}((t_j-x)_+)\bigr)\right|^{\alpha} r^{\beta}
(r+x)^{-1} x^{-\beta} (s-r)^{-\beta} dr dx.
$$
Using the integration formula
$$
\int_0^1 \left(\frac{t}{1-t}\right)^\beta \frac{1}{t+y}\, dt =
\frac{\pi}{\sin \beta \pi}\left[
  1-\left(\frac{y}{1+y}\right)^\beta\right], \ y>0\,,
$$
given on p. 338 in  \cite{gradshteyn:ryzhik:1980}, shows that
\eqref{e:si.check.2} is equivalent to \eqref{eq:SI}.
\end{proof}

\begin{remark} \label{rk:increm.prop}
The increment process 
$$
V_n^{(\alpha,\beta,\gamma)} = Y_{\alpha,\beta,\gamma}(n+1) -
Y_{\alpha,\beta,\gamma}(n), \quad n=0,1,2,\dots\,,
$$
is a stationary \SaS\ process. The argument in Theorem 3.5 in
\cite{owada:samorodnitsky:2015} can be used to check that, in the
case $0<\beta<1$, this
process corresponds to a conservative null operator $T$ in the
sense of \eqref{e:shifted.f}. Furthermore, this process is mixing. See
\cite{rosinski:1995}  and \cite{samorodnitsky:2005} for details. On
the other hand, in the case  {$\beta=0$}, the increment process is an
i.i.d. sequence and, hence, corresponds to a dissipative operator
$T$; recall Remark \ref{rk:boundary}. Furthermore, in the case
$\beta=1$, the increment process is sub-stable, hence corresponds to a
positive operator $T$. In particular, it is not even an ergodic
process. 
\end{remark} 

\section{Some Markov chain theory} \label{sec:markov}

The class of stationary \id\ processes for which we
will prove a functional central limit theorem, is based on a dynamical
system related to Example 5.5 in \cite{owada:samorodnitsky:2015}. In the
present paper we allow the Markov chains involved in the construction
to take values in a space more general than $\bbz$. 

We follow the setup of \cite{chen:2000} and prove additional
 {auxiliary} results we will need in the sequel.  Let  $(Z_n)$
be an irreducible Harris recurrent Markov chain (or simply {\it Harris
chain} in the sequel) on state space $(\X,\XX)$ with transition
probability $P(x,A)$ and  invariant measure $\pi(A)$. Our general
reference for such processes is \cite{meyn:tweedie:2009}. As usually,
we assume that the $\sigma$-field $\XX$ is countably generated. 
We denote by $\P_{\nu}$ the probability law of $(Z_n)$ with initial
distribution $\nu$, and by $\E_{\nu}$  the expectation with respect to
$\P_{\nu}$. 

The collections of sets of finite, and of finite and positive
$\pi$-measure are 
denoted by 
$$
\XX^+:=\{A\in\XX  \text{ such that } \pi(A)>0\}, \ \
\XX^+_0:=\{A\in\XX  \text{ such that }0<\pi(A)<\infty\}\,.
$$
Since the Markov chain is Harris, for any set  $A\in\XX^+$ and any initial
distribution $\nu$, on  {an} event of full probability with respect to
$\P_{\nu}$, the sequence of return times to $A$ defined by 
$\tau_A(0)= 0$ and 
\begin{equation} \label{e:ret.time}
\tau_A(k) = \inf\{n>\tau_A(k-1): Z_n\in A\} \ \ \text{for $k\geq 1$,}
\end{equation}
is a well defined finite sequence. An alternative name for $\tau_A(1)$
is simply $\tau_A$. 

For a set $A\in\XX^+_0$ we denote by 
\begin{equation} \label{e:an.A}
a_n(\nu,A)=\pi(A)^{-1} \sum_{k=1}^{n} \int_\X P^k(x,A) \nu(dx)\,, \ \  n\ge 1
\end{equation}
the mean number of visits to the set starting from initial
distribution $\nu$, up to time $n$, relative to its
$\pi$-measure. When needed, we extend the domain of $a$ to
$[1,\infty)$ by rounding the argument down to the nearest integer. 

An  {\it atom} $\aaa$ of the Markov chain is a subset of $\X$ such that
$P(x,\cdot)=P(y,\cdot)$ for all $x,y\in \aaa$. For an atom the
notation $\P_\aaa$ and $\E_\aaa$ makes an obvious sense. A finite union of atoms
is a set of the form 
\begin{equation} \label{e:union.atoms}
D = \bigcup_{i=1}^q \aaa_i\,, \ q<\infty\,,
\end{equation}
where $\aaa_j\in  \XX^+_0, \, j=1,\ldots, q$ are atoms. 
Any such set  is a special set, otherwise known as a $D$-set, 
see  Definition 5.4 in \cite{nummelin:1984} and \cite{orey:1971},
p. 29. The importance of this fact is that for any two special sets
(and, hence, for any two finite unions of atoms) $D_1$ and $D_2$, 
\be \label{asymptotics}
\lim_{n\to\infty}\frac{a_n(\nu_1,D_1)}{a_n(\nu_2,D_2)}=1
\ee
for any two initial distributions $\nu_1$ and $\nu_2$; see Theorem 2
in Chapter 2 of \cite{orey:1971} or Theorem 7.3 in
\cite{nummelin:1984} (without the assumption of ``speciality''  there
might be $\pi$-small exceptional sets  of initial states). This fact
allows us to use the notation  $a_t := a_t(\nu,D)$ for any arbitrary
fixed special set $D$ and initial distribution $\nu$ when only the
limiting behaviour as $t\to\infty$ of this function is important. For
concreteness, we fix a special set $D$ and use
$\nu(dx)=\pi(D)^{-1}1_D(x)\pi(dx)$. 

A Harris chain is said to be {\it $\beta$-regular}, $0 \leq \beta \leq
1$, if the function $(a_t)$ is  regularly varying at infinity with
exponent $\beta$, i.e.  
\begin{equation*} \label{a_n reg varying}
\lim_{t\to\infty}a_{ct}/a_t=c^\beta \ \ \text{for any $c>0$.}
\end{equation*} 

Let $f:\, \X\to\bbr$ be a measurable function. For a set $A\in\XX^+_0$
the sequence 
\begin{equation*}\label{def:excursion}
\xi_k(A)= \sum_{j=\tau_A(k-1)+1}^{\tau_{A}(k)} f(Z_j), \ k=1,2,\ldots
\end{equation*}
is a well defined sequence of random variables under any law
$\P_\nu$. It is a sequence of i.i.d. random variables under $\P_\nu$ 
if $A$ is an
atom and $\nu$ is concentrated on $A$. 

The following two conditions on a function $f$ will be imposed
throughout the paper.

\begin{equation} \label{e:cond.1}
f\in L^1(\pi)\cap L^2(\pi)  \ \ \text{and} \ \ \int_{\X} f(x) \,
\pi(dx) =0\,,
\end{equation}
 
\begin{equation} \label{e:cond.2}
 \sum_{k=1}^\infty f(\cdot) P^k f(\cdot) \text{ converges in }L^1(\X,\XX,\pi).
\end{equation}
It follows that  
\be\label{variance}
\sigma_f^2:= \int_{\X} f^2(x)\, \pi(dx)+ 2\sum_{k=1}^\infty
\int_{\X} f(x) P^k f(x) \, \pi(dx) < \infty\,.
\ee
We refer the reader to \cite{chen:1999a} and \cite{chen:2000} for a
discussion and examples of functions $f$ satisfying conditions
\eqref{e:cond.1} and \eqref{e:cond.2}. An important implication of the
above assumptions is the following result, proven in Lemma 2.3 of
\cite{chen:1999a}: if $\aaa\in  \XX^+_0$ is an atom such that
$\inf_{x\in\aaa} |f(x)|>0$ (i.e. an $f$-atom), then 
\begin{equation} \label{e:var.excursion}
\E_\aaa \bigl( \xi_1(\aaa)\bigr)^2=\frac{1}{\pi(\aaa)} \sigma_f^2\,.
\end{equation}

We prove next a functional version of Theorem 1.3 in
\cite{chen:2000}. In infinite ergodic theory related results are known
as Darling-Kac theorems; see \cite{aaronson:1981},
\cite{thaler:zweimuller:2006}, and
\cite{owada:samorodnitsky:2015}. The result below can be viewed as a
mean-zero functional Darling-Kac theorem for Harris chains.

\begin{thm}\label{thm:thm1}
Suppose that $(Z_n)$ is a $\beta$-regular
Harris recurrent chain with  $0<\beta\le 1$, 
and suppose $f$ satisfies conditions
\eqref{e:cond.1} and \eqref{e:cond.2}.  Set for $nt\in\N$
\begin{equation*} \label{eq:lin. interp.}
S_{nt}(f)=\sum_{k=1}^{nt} f(Z_k),
\end{equation*}
and for all other $\,t>0$  define $S_{nt}(f)$ by linear interpolation.
Then for any initial distribution $\nu$,  under $\P_\nu$, 
\begin{equation}\label{eq:fclt}
\(\frac{1}{\sqrt{a_n}} S_{nt}(f),\,  t\ge
0\)\ \stackrel{n\to\infty}{\Longrightarrow}\
\((\Gamma(\beta+1))^{1/2}\sigma_f B(\M(t)) ,\,  t\ge 0\),
\end{equation}
weakly in $\CC[0,\infty)$, 
where $B$ is a standard Brownian motion,  which is independent of the
Mittag-Leffler process $\M$.  If $\beta=0$, then the same convergence
holds in finite dimensional distributions. 
\end{thm}

As in \cite{chen:2000}, the proof of Theorem \ref{thm:thm1} proceeds
in three steps: regeneration, the split chain method of
\cite{nummelin:1978},
and finally the use of a geometrically sampled approximation, also
known as a resolvent approximation (see Chapter 5 in \cite{meyn:tweedie:2009}).

In the first step we derive a functional version of a part of Lemma 2.3 in
\cite{chen:2000}, assuming existence of an atom $\aaa\in\XX_0^+$. For
a related result see Theorem 5.1 in \cite{kasahara:1984}. 

We define the discrete local time at $\aaa$ by
\begin{equation*}
\ell_{\aaa}(n):= \max\{k\ge 0: \tau_\aaa(k)\le n\}
\end{equation*}
Then  the sequence
\be\label{eq:iid}
\left( \bigl(\xi_k(\aaa), \tau_\aaa(k)-\tau_\aaa(k-1)\bigr), \
k=1,2,\ldots\right) \text{ \ is  i.i.d under $\P_\aaa$.}
\ee
\begin{lem}\label{lemma:atom}
The convergence \eqref{eq:fclt} holds 
under the assumptions of Theorem \ref{thm:thm1}, with the additional
assumption that  {$(Z_n)$}   has an $f$-atom $\aaa$. 
\end{lem}
\begin{proof}
Unless stated otherwise, all the distributional statements below are
understood to be under $\P_\nu$, for an arbitrary fixed initial
distribution $\nu$. Let
 $$
\phi(n)=\sum_{k=1}^{n} P^k(\aaa,\aaa), \ n=1,2,\ldots\,,
$$
and note that by \eqref{asymptotics},
\begin{equation} \label{e:phi.a}
\lim_{n\to\infty}  \phi(n)/a_n = \pi(\aaa)\,.
\end{equation}

Let $\phi^{-1}$ denote an asymptotic inverse of $\phi$. By Lemma 3.4 in \cite{chen:1999b}, as $n\to\ff$, the stochastic
process 
$$
T_n(t):=\tau_\aaa(\lfloor nt\rfloor)/\phi^{-1}(n), \ t\ge 0
$$ 
converges weakly in the $J_1$-topology on $\DD[0,\infty)$ to a
$\beta$-stable subordinator with the Laplace transform
$\exp\{-t \theta^\beta/\Gamma(\beta+1)\}$ for
\mbox{$0<\beta<1$}, and to a line with slope one when  $\beta=1$ (we
will deal with $\beta=0$ in a  moment). Setting
\be\label{eq:brownian component}
W_{n}(t)=n^{-1/2}\sum_{k=1}^{nt}\xi_k(\aaa), \ t\geq 0
\ee
(defined for fractional values of $nt$ by linear interpolation), the laws  of
$\bigl\{ \bigl( (W_{n}(t),\, t\geq 0), \, (T_n(t), \, t\geq 
0)\bigr)\bigr\}_{n\in\N}$  are tight in 
$\CC[0,\ff) \times \DD[0,\ff)$ since the marginal laws converge weakly
in the corresponding spaces as  $n\to\infty$.
By \eqref{eq:iid} every subsequential limit is a bivariate L\'evy
process, with one marginal process a Brownian motion, and the other
marginal process a subordinator. By   the L\'evy-It\^o decomposition,
the  Brownian and subordinator  components must be independent, so all
subsequential limits coincide, and the entire bivariate sequence converges weakly in $\CC[0,\ff) \times
\DD[0,\ff)$ to a bivariate L\'evy process with independent marginals. 

If  $0<\beta<1$, weak convergence of $(T_n(t))$ is easily seen to
imply, by inversion,  
finite dimensional convergence, as $n\to\infty$,  
of the sequence
$(\ell_\aaa(nt)/ \phi(n))$ (defined for fractional $nt$ by linear
interpolation) to $\Gamma(\beta+1) M_\beta$. 
Since the paths are increasing, and the limit is continuous, this
guarantees convergence in  $\DD[0,\ff)$; see \cite{bingham:1971} and
\cite{yamazato:2009}.  Similarly, we obtain weak 
convergence to a line with slope one for $\beta=1$. 

In the case $\beta=0$, by Theorem 2.3 of \cite{chen:1999b} and
\eqref{e:phi.a}, we see that the sequence of processes $(\ell_\aaa(nt)/
\phi(n))$ 
converges in finite-dimensional distributions to a  limit, equal
to zero at $t=0$ and consisting of the same standard 
exponential random variable repeated for all $t>0$.
Moreover, by  Lemma 2.3 in \cite{chen:2000}, the exponential random
variable is independent of the limiting Brownian motion in 
\eqref{eq:brownian component}, when we perform a subsequential limit
scheme for the bivariate process  $\bigl\{
\bigl( (W_{n}(t),\, t\geq 0), \, (\ell_\aaa(nt)/
\phi(n))\bigr)\bigr\}_{n\in\N}$,  similarly to the above. This time
the convergence is in finite-dimensional distributions. 

Suppose now that $0<\beta\leq 1$. 
If $\DD_+[0,\ff)$ denotes the subset of $\DD[0,\ff)$ consisting of
nonnegative functions, then the composition map $(x,y)
\longrightarrow x \circ y$ from $\CC[0,\ff) \times
\DD_+[0,\ff)$ to $\DD[0,\ff)$ is continuous at a point $(x,y)$ if $y$
is continuous. It follows, therefore, by
the continuous mapping theorem  that 
\begin{equation} \label{e:comp}
\(\frac{1}{\sqrt{\phi(n)}}\sum_{k=1}^{\ell_\aaa(\lfloor nt\rfloor)} \xi_k(\aaa) , t\ge
0\)\stackrel{n\to\infty}{\Longrightarrow}\(\bigl(
\Gamma(\beta+1)\E_\aaa\xi_1(\aaa)^2\bigr)^{1/2}  B(\M(t)) , t\ge 0\)
\end{equation}
in $\DD[0,\ff)$, with $(M_{\beta}(t))$  independent of $(B(t))$ in the
right hand side. By \eqref{e:phi.a} we obtain  
\begin{equation*}
\(\frac{1}{\sqrt{a_n}}\sum_{k=1}^{\ell_\aaa(\lfloor nt\rfloor)} \xi_k(\aaa) , t\ge
0\)\stackrel{n\to\infty}{\Longrightarrow}\((\pi(\aaa) \Gamma(\beta+1)\E_\aaa\xi_1(\aaa)^2)^{1/2}
B(\M(t)) , t\ge 0\).
\end{equation*}
Recalling \eqref{e:var.excursion}, we have shown that 
\begin{equation}\label{eq:lem2.3}
\(\frac{1}{\sqrt{a_n}}\sum_{k=1}^{\ell_\aaa(\lfloor nt\rfloor)} \xi_k(\aaa) , t\ge
0\)\stackrel{n\to\infty}{\Longrightarrow}\((\Gamma(\beta+1))^{1/2}\sigma_f
B(\M(t)) , t\ge 
0\)
\end{equation}
in $\DD[0,\ff)$. 

We now proceed to relate \eqref{eq:lem2.3} to the statement of the
lemma. First of all, Corollary 3 in \cite{zweimuller:2007a} allows us
to simplify the situation and assume that the chain starts at the
$f$-atom $\aaa$. We can write 
\begin{equation}\label{eq:splitsum}
\(\frac{1}{\sqrt{a_n}}\sum_{k=1}^{ nt} f(Z_k),\, t\ge
0\) = \left( \frac{1}{\sqrt{a_n}}\sum_{k=1}^{\ell_\aaa(\lfloor
    nt\rfloor)} \xi_k(\aaa) \right.
\end{equation}
$$
\left.
+\frac{1-nt+\lfloor nt\rfloor}{\sqrt{a_n}}\mathop{\sum_{k=}^{\lfloor
nt\rfloor}}_{\tau_\aaa(\ell_\aaa(\lfloor nt\rfloor) )+1} f(Z_k)
+\frac{nt-\lfloor nt\rfloor}{\sqrt{a_n}}\mathop{\sum_{k=}^{\lfloor
nt\rfloor+1}}_{\tau_\aaa(\ell_\aaa(\lfloor nt\rfloor) )+1} f(Z_k)
,\  t\ge 0\right).
$$
Since the convergence in \eqref{eq:lem2.3} occurs in the $J_1$
topology on $\DD[0,\ff)$ and the limit is continuous (recall that we
are considering 
the case $0<\beta\leq 1$), in order to prove convergence of the
processes in \eqref{eq:splitsum} in $\CC[0,\ff)$, 
 we need only show that the second and the third terms in the right
 hand side of \eqref{eq:splitsum} are negligible in $\CC[0,\ff)$. 
 We treat in details the second term; the third term can be treated
 similarly. Restricting ourselves to $\CC[0,1]$, we will prove that 
\begin{equation*}\label{eq:zero in prob}
\frac{1}{\sqrt{a_n}}\sup_{0\le t\le 1} \left| \mathop{\sum_{k=}^{\lfloor
nt\rfloor}}_{\tau_\aaa(\ell_\aaa(\lfloor nt\rfloor))+1} f(Z_k)\right|
\stackrel{n\to\infty}{\longrightarrow} 0
\end{equation*}
in probability. To this end we rewrite this expression as 
$$
\frac{1}{\sqrt{a_n}}\max_{m=0,\ldots,n} \left|
  \mathop{\sum_{k=}^{m}}_{\tau_\aaa(\ell_\aaa(m))+1} f(Z_k) \right| \le 
\frac{1}{\sqrt{a_n}}\max_{j=0,\ldots,\ell_\aaa(n)}\max_{\stackrel{m=}{\tau_\aaa(j)+1,\ldots,
    \tau_\aaa(j+1)}}\left| \mathop{\sum_{k=}^{m}}_{\tau_\aaa(j)+1}
f(Z_k) \right| .
$$
Letting $W_j$ denote the inner maximum, we must show that for any $\epsilon>0$, choosing $n$ large enough implies
\begin{equation*}
\P_\aaa\(\max_{j=0,\ldots,\ell_\aaa(n)}W_j>\epsilon \sqrt{a_n}\) <\epsilon.
\end{equation*}
Since the sequence $\bigl( \ell_\aaa(n)/a_n\bigr)$ converges weakly,
it is tight, and there exists $M_\eps$  such that for all $n$, $\P_\aaa\(\ell_\aaa(n)>
M_\eps a_n\)<\epsilon/2$. 
 Thus we need only check that for $n$ large enough
\begin{eqnarray}\label{eq:eps over 2}
\P_\aaa\( \max_{0\leq j \leq M_\eps a_n}W_j >\epsilon \sqrt{
  a_n}\) &=&1-\P_\aaa\( \max_{0\leq j \leq M_\eps a_n}W_j \le
             \epsilon \sqrt{a_n}\)\\ 
&=& 1- \(1-\P_\aaa\(W_1^2> \epsilon^2 a_n\)\)^{\lfloor M_\eps
    a_n \rfloor + 1}<\epsilon/2 \nn. 
\end{eqnarray}
To see this we use Lemma 2.1 in \cite{chen:2000} which states that
under our assumptions we have $\E_\aaa W_1^2<\infty.$ Thus,
\begin{equation*}\label{eqn2:excursion to 0}
\P_\aaa\(W_1^2> \epsilon^2 a_n\)= o\bigl( a_n^{-1}\bigr)  \ \text{
  as }\ n\to\infty\,, 
\end{equation*}
which verifies \eqref{eq:eps over 2}. This proves the lemma in the case
$0<\beta\leq 1$. 

If $\beta=0$, then the same argument starting with \eqref{e:comp}
works. The argument is  easier in this case since we only need to prove
convergence in finite-dimensional distributions. We omit the details. 
\end{proof}


We are now ready to prove Theorem \ref{thm:thm1} in general. 
\begin{proof}[Proof of Theorem \ref{thm:thm1}]
As before, the  {distributional} statements below are
understood to be under $\P_\nu$, for an arbitrary fixed initial
distribution $\nu$. 
The split chain method of \cite{nummelin:1978} allows us to
extend the result from the situation in Lemma \ref{lemma:atom}, where
we assumed the existence of an $f$-atom, 
to the case where we only assume that there exists a $C\in\XX^+_0$ such that
\begin{equation}\label{assumption}
\inf_{x\in C}|f(x)|>0\quad\text{and}\quad P(x,A) \ge b
1_{C}(x)\pi_C(A) \quad x\in \X, \ A\in\XX 
\end{equation}
for some $0<b\le 1$ where $\pi_C(\cdot) :=\pi(C)^{-1}\pi(C\cap\cdot)$.

A very brief outline of the split chain method is as follows (see
\cite{nummelin:1978} or Chapter 5 in \cite{meyn:tweedie:2009} for more
details). One can enlarge the probability space in order to obtain an
extra sequence of Bernoulli random variables $(Y_n)$. 
The $(Y_n)$ are chosen so that the {\it split chain} $(Z_n,Y_n)$ is a
Harris chain on $\X\times\{0,1\}$ 
and such that $C\times\{1\}$ is an $f$-atom (where $f$ is extended to
$\X\times\{0,1\}$  in the natural way). 
Moreover, this can be done so that conditions \eqref{e:cond.1} and
\eqref{e:cond.2} continue to  hold for the split chain,  and also 
\eqref{variance} holds for $\tilde P$ and $\tilde\pi$ (the 
transition kernel and invariant measure for the split chain.) Then an
application of Lemma \ref{lemma:atom} to the split chain proves the 
claim of the theorem under the assumption \eqref{assumption}. 

The final step is to get rid of assumption \eqref{assumption}, so
we no longer assume that \eqref{assumption} holds to start with. 
We follow the usual procedure which, for (a small) $p>0$  uses the
renewal process 
\begin{equation*}
N(t):= \max\{n\geq 1:\Gamma_n\le t\}, \ t\geq 0\,,
\end{equation*}
with i.i.d. renewal intervals $(\Gamma_{n+1}-\Gamma_n)$, which have
the geometric distribution 
\be\label{geometric}
\P(\Gamma_1=k):=(1-p)p^{k-1}, \ k=1,2,\ldots\,.
\ee
The idea is to approximate the original chain $(Z_n)$ by its resolvent
chain $(Z_{\Gamma_k})$, where $(\Gamma_k)$ are as in
\eqref{geometric}, and independent of $(Z_n)$. We then let  $p\to 0$. 

The resolvent chain is just $(Z_n)$ observed at the negative binomial
renewal times $(\Gamma_k)$  (this chain is also called a
geometrically sampled chain). Its transition kernel is
$$
P_p(x,A):=(1-p)\sum_{k=1}^\infty p^{k-1} P^k(x,A),
$$
and, clearly, $\pi$ is still an invariant measure.
 For the resolvent chain, the assumptions \eqref{e:cond.1} and
 \eqref{e:cond.2} allow one to define, similarly to \eqref{variance}, 
\begin{align*}
\sigma_{p,f}^2 &:= \int_{\X} f^2(x)\pi(dx)+ 2\sum_{k=1}^\infty \int_{\X}
f(x) P_p^k f(x) \pi(dx) \\[5pt]
&= \int_{\X} f^2(x)\pi(dx)+ 2(1-p)\sum_{k=1}^\infty
\int_{\X} f(x) P^k f(x) \pi(dx).
\end{align*}
Furthermore,  the resolvent chain is $\beta$-regular if the original
chain is, and the sequence $(a_n^{(p)}$ corresponding to the resolvent
chain (see the discussion following \eqref{asymptotics}) satisfies 
$$
a_n^{(p)} \sim (1-p)^{1-\beta}a_n, \ n\to\infty\,;
$$
see (4.26) in \cite{chen:2000}. The latter paper also shows that 
the resolvent chain $(Z_{\Gamma_k})$
satisfies \eqref{assumption} (see also Theorem 5.2.1 in
\cite{meyn:tweedie:2009}).  

Suppose that $0<\beta\leq 1$. 
Since we have already proved the theorem under the assumption
\eqref{assumption}, we can use  Theorem 2.15(c) in
\cite{jacod:shiryaev:1987},  and the ``converging together lemma'' in
Proposition 3.1 of 
\cite{resnick:2007} to obtain 
\begin{align*}
&\[\(\frac{1}{\sqrt{a_n}}\sum_{k=1}^{nt} f(Z_{\Gamma_k}), \ t\ge 0\),\(\frac 1 n N(nt), \ t\ge 0\)\] \srlim{\Longrightarrow}  \\[5pt]
&\[\(\sqrt{(1-p)^{1-\bb}\, \Gamma(\beta+1)}\,  {\sigma_{p,f}}B(\M(t)),\ t\ge 0\),\bigl((1-p)t,\ t\ge 0\bigr)\]
\end{align*}
in $\CC[0,\infty)\times\DD[0,\infty)$.  As before, it is legitimate to
apply the continuous mapping theorem, to obtain 
\begin{equation} \label{e:conv.D}
\(\frac{1}{\sqrt{a_n}}\sum_{k=1}^{N(nt)} f(Z_{\Gamma_k}), \ t\ge 0\) \srlim{\Longrightarrow}
\(\sqrt{(1-p)^{1-\bb} \, \Gamma(\beta+1)}\, \sigma_{p,f}B\(\M((1-p)t)\),\ t\ge 0\)
\end{equation}
in $\DD[0,\infty)$. Repeating the same argument with the continuous
version of the counting process $(N(t))$, given by
$$
N_c(t) =  {N(t)} + \frac{t-\Gamma_n}{\Gamma_{n+1}-\Gamma_n} \ \ \text{for
  $\Gamma_n\leq t\leq \Gamma_{n+1}$, $n=0,1,\ldots$,}
$$
leads to 
\begin{equation} \label{e:conv.C}
\(\frac{1}{\sqrt{a_n}}\sum_{k=1}^{N_c(nt)} f(Z_{\Gamma_k}), \ t\ge 0\) \srlim{\Longrightarrow}
\(\sqrt{(1-p)^{1-\bb} \, \Gamma(\beta+1)}\, \sigma_{p,f}B\(\M((1-p)t)\),\ t\ge 0\)\,,
\end{equation}
this time in $\CC[0,\infty)$. We now show that the convergence
statement in \eqref{e:conv.C} is sufficiently close to the required
convergence statement in \eqref{eq:fclt}, and for this purpose the
$\DD[0,\infty)$ version in \eqref{e:conv.D} will be useful.

We will restrict ourselves to the  {interval} $[0,1]$. Since 
\be\nn
\sqrt{1-p}\ \sigma_{p,f} \longrightarrow \sigma_f\quad\text{as}\ \ p\to0\,,
\ee
the second converging together theorem (see Theorem 3.5 in
\cite{resnick:2007}), says that \eqref{eq:fclt} will follow once we
check that for any $\eps> 0$,
\begin{equation} \label{e:link.conv}
\lim_{p\to 0} \limsup_{n\to\infty} \P_\nu\(\sup_{0\le t\le 1}
\frac{1}{\sqrt{a_n}}\lb \sum_{k=1}^{N_c(nt)}
f(Z_{\Gamma_k})-\sum_{k=1}^{nt} f(Z_k) \rb>\eps\) = 0\,. 
\end{equation}
To this end we bound the probability in the left hand side of
\eqref{e:link.conv} by a sum of  {3} probabilities: 
$$
\P_\nu\(\sup_{0\le t\le 1}
\frac{1}{\sqrt{a_n}}\lb \sum_{k=1}^{N_c(nt)}
f(Z_{\Gamma_k})-\sum_{k=1}^{N(nt)} f(Z_{\Gamma_k}) \rb>\frac{\eps}{3}\) 
+ \P_\nu\(\sup_{0\le t\le 1}
\frac{1}{\sqrt{a_n}}\lb \sum_{k=1}^{N(nt)}
f(Z_{\Gamma_k})-\sum_{k=1}^{\lfloor nt\rfloor} f(Z_k) \rb>\frac{\eps}{3}\) 
$$
\begin{equation} \label{e:decomp3}
+ \P_\nu\(\sup_{0\le t\le 1}
\frac{1}{\sqrt{a_n}}\lb \sum_{k=1}^{\lfloor nt\rfloor}
f(Z_k)-\sum_{k=1}^{nt} f(Z_k) \rb>\frac{\eps}{3}\) \,.
\end{equation}
Keep for a moment $0<p<1/2$ fixed. Note that  
$$
\limsup_{n\to\infty} \P_\nu\(\sup_{0\le t\le 1}
\frac{1}{\sqrt{a_n}}\lb \sum_{k=1}^{N_c(nt)}
f(Z_{\Gamma_k})-\sum_{k=1}^{N(nt)} f(Z_{\Gamma_k}) \rb>\frac{\eps}{3}\) 
$$
$$
\leq \limsup_{n\to\infty} \P_\nu\( \frac{1}{\sqrt{a_n}} \max_{k\leq
  2n} |f(Z_k)|>\frac{\eps}{3}\)  = 0
$$
because, eventually, $N_c(n)\leq 2n$ and the convergence in
\eqref{e:conv.D} is to a continuous limit. A similar argument shows
that for a fixed $0<p<1/2$, 
$$
\limsup_{n\to\infty} \P_\nu\(\sup_{0\le t\le 1}
\frac{1}{\sqrt{a_n}}\lb \sum_{k=1}^{\lfloor nt\rfloor}
f(Z_k)-\sum_{k=1}^{nt} f(Z_k) \rb>\frac{\eps}{3}\) =0\,.
$$
It remains to handle the middle probability in \eqref{e:decomp3}. 
Let $\delta_k=1$ at the renewal times and $0$
otherwise.  By the self-similarity of the Brownian motion and the
Mittag-Leffler process, the convergence statement in \eqref{e:conv.D} 
 can be rewritten as 
$$
\(\frac{1}{\sqrt{a_n}}\sum_{k=1}^{\lfloor nt\rfloor}\delta_k f(Z_{k}), \ t\ge 0\) \srlim{\Longrightarrow}
\bigl(\sqrt{(1-p)\, \Gamma(\beta+1)}\,\sigma_{p,f}B(\M(t)),\ t\ge 0\bigr).
$$
Replacing $p$ by $1-p$ and, hence, each $\delta_k$ by $1-\delta_k$,
gives us  also 
\begin{equation*}\label{eq:1-delta}
\(\frac{1}{\sqrt{a_n}}\sum_{k=1}^{\lfloor nt\rfloor}(1-\delta_k) f(Z_{k}), \ t\ge 0\) \srlim{\Longrightarrow}
\bigl(\sqrt{p \, \Gamma(\beta+1)}\,\sigma_{1-p,f}B(\M(t)),\ t\ge 0\bigr).
\end{equation*}
Both of these weak convergence statements take place in the $J_1$
topology on $\DD[0,\ff)$. 
Therefore, 
\be\nn
\P\(\sup_{0\le t\le 1} \frac{1}{\sqrt{a_n}}\lb \sum_{k=1}^{N(nt)} f(Z_{\Gamma_k})-\sum_{k=1}^{\lfloor nt\rfloor}\delta_k f(Z_k)
\rb>\frac{\eps}{3}\)\srlim{\longrightarrow}\P\(\sup_{0\le t\le
  1}\sqrt{p {\Gamma(\beta+1)}}\,\sigma_{1-p,f}\lb B(\M(t))\rb >\frac{\eps}{3}\)\,, 
\ee
which goes to $0$ as $p\to 0$. This completes the proof in the case
$0<\beta\leq 1$. Once again, the case $\beta=0$ is similar but easier,
since we are only claiming finite-dimensional weak convergence.
\end{proof}
%

\begin{remark} \label{rk:l2}
If in Theorem \ref{thm:thm1} the function $f$ is supported by a finite
union of atoms $D$, then we also have 
\begin{equation} \label{e:l2bound}
\sup_{n\geq 1} \, \E_\nu \sup_{0\leq t\leq L} \left(
  \frac{1}{\sqrt{a_n}}\,   S_{nt}(f)\right)^2 <\infty
\end{equation} 
for the initial distribution $\nu(dx)=\pi(D)^{-1}1_D(x)\pi(dx)$ 
and any $0<L<\infty$. 
To see this, it is enough to consider the case where the initial
distribution $\nu$ is given, instead, by
$\nu(dx)=\pi(\aaa)^{-1}1_\aaa(x)\pi(dx)$, where $\aaa$ is a single
atom of positive measure, forming a part of $D$.  We use the notation
in Lemma \ref{lemma:atom}. It is elementary to check that for each
$n\geq 1$, 
\begin{equation*}
\hat \ell_{\aaa}(n):= \min\{k\ge 0: \tau_\aaa(k)> n\} =
\ell_{\aaa}(n)+1 
\end{equation*} 
is a stopping time with respect to the discrete time filtration 
$$
\FF_k =\sigma \bigl( \xi_j(\aaa), \, \tau_\aaa(j), \, j=1,\ldots,
k\bigr), \ k=1,2,\ldots\,,
$$
while the process 
$$
\sum_{j=1}^k  \xi_j(\aaa), \ k=1,2,\ldots
$$
is a martingale with respect to the same filtration. Increasing $L$, if necessary, to make it an integer, we see that
\begin{equation}  \label{e:mg.doob}
\E_\aaa \sup_{0\leq t\leq L} \left(
  \frac{1}{\sqrt{a_n}}\,   S_{nt}(f)\right)^2
\leq  \frac{2}{a_n} \E_\aaa \max_{m=1,\ldots, \hat \ell_{\aaa}(\lfloor
  nL\rfloor)} \left( \sum_{j=1}^m  \xi_j(\aaa)\right)^2\,.
\end{equation}
By Doob's
inequality and the optional stopping theorem,  this can be further bounded by 
$$
\frac{8}{a_n} \E_\aaa \left( \sum_{j=1}^{\hat \ell_{\aaa}(\lfloor
  nL\rfloor)}  \xi_j(\aaa)\right)^2 = 8\E_\aaa (\xi_1(\aaa))^2\, 
\frac{\E_\aaa \, \hat \ell_{\aaa}(\lfloor
  nL\rfloor)}{a_n}\,.
$$
Since 
$$
\E_\aaa (\xi_1(\aaa))^2  <\infty \ \ \text{and} \ \ \sup_{n\geq 1} \frac{\E_\aaa \, \hat
  \ell_{\aaa}(n)}{a_n}<\infty 
$$
by the assumption and the discussion at the beginning of the proof of
Lemma \ref{lemma:atom}, the claim \eqref{e:l2bound} follows. 
\end{remark}

We will also need a version of Theorem \ref{thm:thm1}, in which the
initial distribution is not fixed but, rather,  diffuses, with $n$,
over the set  $\{\tau_{D} \leq n\}$. We will only  {consider} the case of
a finite union of atoms $D = \bigcup_{i=1}^q \aaa_i \in \XX^{+}_0$.  

We first reformulate our Markovian setup in the language of standard
infinite ergodic theory. Let $E =\X^\bbn$ be the path space
corresponding to the Markov chain. We equip $E$ with the 
usual cylindrical $\sigma$-field $\EE =\XX^\bbn$. Let $T$ be
the left shift operator on the path space $E$, i.e.  
$T(\bx)=  
(x_2,x_3,\ldots)$ 
for $\bx=(x_1,x_2,\ldots)\in E$.  Note that $T$
preserves the measure $\mu$ on $E$ defined by 
\be\label{def:mu}
\mu(A):=\int_\X \P_x(A) \,\pi(dx), \text{ for events } A\in\EE
\ee
(as usually, the notation $\P_x$ refers to the initial distribution
$\nu=\delta_x$, $x\in\X$.) Notice that the measure $\mu$ is infinite
if the invariant measure $\pi$ is. This is, of course, always the
case if $0\leq \beta<1$.  In the sequel we will usually assume that
$\pi$ is infinite even when $\beta=1$. 

We will need certain ergodic-theoretical properties of the quadruple
$(E, \EE, \mu, T)$. 
As shown in \cite{aaronson:lin:weiss:1979}, $T$ is \textit{conservative} and \textit{ergodic}; this implies that
$$
\sum_{n=1}^{\infty} 1_A \circ T^n = \infty \ \ \text{$\mu$-a.e. on }
E  
$$
for every $A \in \mathcal{E}$ with $\mu(A) > 0$.  For a finite union
of atoms $D\in\XX^+_0$ as in \eqref{e:union.atoms}, let  
\be\label{def:d omega}
{\tilde{D}}:=\{\bx\in E:\,  x_1\in D\}
\ee
be the set of paths which start in $D$. The first return time to $D$
as defined in \eqref{e:ret.time} can then be viewed as a function on
the product space $E$ via  
$$
\tau_{D}(\bx) = \inf \{ n \geq 0: \, T^n\bx\in \tilde D\} = 
\inf \{ n \geq 0:\,  x_n \in D  \}\,, \ \ \ \bx=(x_1,x_2,\dots) \in E\,.
$$
The \textit{wandering rate sequence} (corresponding to the
set $\tilde D$) is the sequence $\mu(\tau_{D} \leq n)$, $n=1,2,\ldots$. 
Since $T$ is measure-preserving, this is a finite sequence. Since the
Markov chain  is $\beta$-regular, this sequence turns out to be
regularly varying as well, as we show below. For the
ergodic-theoretical notions used in the proof see
\cite{aaronson:1997} and \cite{zweimuller:2009}.  

\begin{lem}  \label{l:RV.wander}
Suppose that $(Z_n)$ is a $\beta$-regular
Harris recurrent chain with  $0\leq \beta\le 1$, with an infinite
invariant measure $\pi$. Let $D$ be a
finite union of atoms. 
Then the wandering rate $\mu(\tau_{{D}} \leq n)$ is a regularly
varying sequence of exponent $1-\beta$. More precisely, 
$$
\mu(\tau_{{D}}\le n) \sim
\frac{1}{\Gamma(1+\beta)\Gamma(2-\beta)}\frac{n}{a_n}\ \ \text{as
  $n\to\infty$.} 
$$
\end{lem}
\begin{proof}
Let $Q$ be a Markov semigroup on $\X$ which is dual to $P$ with
respect to the measure $\pi$. That is, 
for every $k=1,2,\ldots$ and every bounded measurable function $f:\,
\X^k \to \bbr$,
$$
\int_\X \pi(dx_1)\int_\X P(x_1,dx_2)\ldots \int_\X
P(x_{k-1},dx_k)f(x_1,\ldots, x_k)
$$
$$
= \int_\X \pi(dx_k)\int_\X Q(x_k,dx_{k-1})\ldots \int_\X
Q(x_{2},dx_1)f(x_1,\ldots, x_k)\,.
$$
Since $\pi$ is an invariant measure for $P$, it is also invariant for $Q$. Define a $\sigma$-finite measure on $(E,\EE)$ analogous to $\mu$ in \eqref{def:mu}, but using $Q$ instead of $P$, i.e.
$$
\hat \mu(A) :=\int_\X \Q_x(A) \, \pi(dx), \text{ for events } A\in\EE\,.
$$
We claim  that the set $\tilde D$ given in \eqref{def:d omega} is a
Darling-Kac set for the shift  operator $T$ on the space $(E,
\EE, \hat \mu)$. According to the definition of a Darling-Kac
set (see Chapter 3 in \cite{aaronson:1997}), it is enough to show 
that  
\begin{equation}  \label{e:DK.cond}
\frac{1}{a_n} \sum_{k=1}^n \hat T_Q^k 1_{\tilde D}(\bx) \to \hat \mu(\tilde D) = \pi(D) \ \ \text{uniformly $\hat\mu$-a.e. on } \tilde D,
\end{equation}
where $\hat T_Q: L^1(\hat \mu) \to L^1(\hat \mu)$ is the dual operator
defined by  
$$
\hat T_Q g(\bx) := \frac{d(\hat \mu_g \circ T^{-1})}{d\hat \mu}\,
(\bx) 
$$
with 
$$\hat \mu_g(A) := \int_A g(\bx) \, \hat \mu(d\bx), \ A\in\EE
$$
a signed measure on $(E, \EE)$, absolutely continuous with
respect to $\hat\mu$. 
Note that the dual operator $\hat T_Q$ satisfies $\hat T^k_Q
1_{{\tilde{D}}}(\bx) = P^k(x_1,D)$; see Example 2 in
\cite{aaronson:1981}. Since we may choose $a_n$ as in
\eqref{e:an.A} with $A=D$ (recall that a finite union of atoms is a
special set), it is elementary to check that
\eqref{e:DK.cond} holds, and the 
uniformity of the convergence stems from the fact that the left
side in \eqref{e:DK.cond}  takes at most $q$ different
values on $\tilde D$.  Applying Proposition 3.8.7 in
\cite{aaronson:1997}, we obtain 
$$
\hat \mu(\tau_{D}\le n) \sim \frac{1}{\Gamma(1+\beta)\Gamma(2-\beta)}\frac{n}{a_n}\,.
$$
However, by duality,
$$
\mu(\tau_{D}\le n)  = \hat \mu(\tau_{D}\le n)\,.
$$
Since $(a_n)$ is regularly varying with exponent $\beta$, the exponent
of regular variation of  the wandering sequence is, obviously,
$1-\beta$. 
\end{proof}
 
\begin{remark}
Referring to the proof of Lemma \ref{l:RV.wander}, 
it should be noted that using in \eqref{def:d omega} 
a set $D\in\XX^+_0$ different from a finite union of atoms, may still
define a set $\tilde D$ that is a  Darling-Kac set. For  example,
suppose that $(Z_k)$ is a 
random walk on $\R$ with standard Gaussian steps; that is,
$$
P(x,B) = \P(G\in B-x), \ x\in \R, \ \ \text{$B$ Borel.}
$$
 Here $G\sim N(0,1)$. In this case the Lebesgue measure $\pi$ on
 $\bbr$  is an invariant measure. 
It is not hard to see that $D=[0,1]$ is a  special set
that is not a finite union of atoms. Further, $a_n \sim \sqrt{n/2\pi}$
as $n\to\infty$. 
We claim that the set $\tilde D$ of paths starting in $[0,1]$,
is a Darling-Kac set  for the conservative measure-preserving shift
operator on $E$. To  see this note that, in this case, there is no
difference 
between the semigroup $P$ and the dual semigroup $Q$, and, hence, 
$$
\frac{1}{\sqrt{n}}\sum_{k=1}^n \hat T^k1_{\tilde D} =
\frac{1}{\sqrt{n}}\sum_{k=1}^n P^k1_{[0,1]}\to  \frac{1}{\sqrt{2\pi}}
$$
uniformly on $[0,1]$.
\end{remark}

We can view the sums $S_n(f)$ as being defined on the path space
$E$ by setting 
$h(\bx) := f(x_1)$, $\bx=(x_1,x_2,\dots) \in
E$, and then writing 
$$
\hat S_n(f)(\bx)  = \sum_{k=1}^n h \circ T^k (\bx) =
\sum_{k=1}^n f(x_k) \,. 
$$
This defines the notation used in Theorem \ref{cor:entrance time}
below. 
Define a sequence of probability measures $(\mu_n)$ on $E$ by 
$$
\mu_n(A):=\frac{\mu(A\cap\{\tau_{{D}}\le n\})}{\mu(\tau_{{D}}\le n)}, \ A\in\EE.$$
\begin{thm}\label{cor:entrance time}
Suppose that, in addition to the hypotheses of Theorem
\ref{thm:thm1},  $f$ is supported by a finite union of atoms $D =
\bigcup_{i=1}^q \aaa_i \in \XX^{+}_0$. Then for every $L>0$, 
under the measures $\mu_{nL}$,  
\be\nn
 \(\frac{1}{\sqrt{a_n}} \hat S_{nt}(f),\,  0\leq t\leq L\)\
\stackrel{n\to\infty}{\Longrightarrow}
\((\Gamma(\beta+1))^{1/2}\sigma_f B(\M(t-T_\infty^{L})) ,\,  0\leq
t\leq L\),
\ee 
where $T_\infty^{L}$ is independent of the process $B(M_\beta(t))$ and
$\P(T_\infty^{L}\le x)=\left(x / L \right)^{1-\beta}$ for
$x\in[0,L]$ (in particular, $T_\infty^{L}=0$ a.s. if $\beta=1$). The
convergence is weak convergence in $\CC[0,L]$ for $0<\beta\leq 1$ and
convergence in
finite-dimensional distributions if $\beta=0$. Furthermore, for all
$0\leq \beta\leq 1$,
\begin{equation} \label{e:miss.lemma}
\sup_{n\ge 1}\int_E \(\frac{\hat S_n(f)}{\sqrt{a_n}}\)^2\, d \mu_n <\infty\,.
\end{equation}
\end{thm}

\begin{remark} \label{rk:higher.mom}
By the similar argument to that in Remark \ref{rk:l2}, it is not hard to check that, under the assumptions of Theorem
\ref{cor:entrance time}, if for some $p>2$, $\E_\aaa
\bigl|\xi_1(D)\bigr|^p<\infty$  for any atom $\aaa$  {constituting} $D$,
then we, correspondingly, have 
$$
\sup_{n\ge 1}\int_E \(\frac{|\hat S_n(f)|}{\sqrt{a_n}}\)^p\, d \mu_n
<\infty\,.
$$
\end{remark}

\begin{proof} 
We prove convergence of the finite-dimensional distributions first. 
For typographical convenience we will only consider one-dimensional
distributions. In the case of more than one dimension, the argument is
similar, but the notation is more cumbersome. 
During this proof we may and will modify the definition of
$S_{ nt  }(f)$ to have the sum starting at $k=0$. Suppose first that $L=1$.

Set for $m=1,2\ldots$, $x\in\X$ and $i=1,\ldots, q$
$$
p_m(x,i):=\P_x(Z_{\tau_{{D}}}=\aaa_i|\tau_{{D}}=m)\,.
$$
Then  for $\lambda \in \bbr$ and a large $K=1,2,\ldots$, 
\begin{align}
&\mu_{n}\( \frac{\  \hat S_{nt}(f)}{\sqrt{a_n}}>\lambda \) \label{eq:asymptotic sum} \\[5pt]
&=\sum_{m=0}^{n} \frac{1}{\mu(\tau_{{D}}\le n)}\int_{\X} \P_x\(
 \frac{  S_{
    nt }(f)}{\sqrt{a_n}}>\lambda, \tau_{{D}}=m \) \pi(dx)  \notag \\[5pt]
&=\sum_{m=0}^{n} \int_{\X} \frac{\P_x(\tau_{{D}}=m)}{\mu(\tau_{{D}}\le n)}
\sum_{i=1}^q p_m(x,i)\P_{\aaa_i}\(  \frac{
  S_{(nt -m)_+}(f)}{\sqrt{a_n}}>\lambda \)  \pi(dx) \nn  \\[5pt]
&= \sum_{k=1}^{K} \sum_{m=\lfloor {(k-1)n/K}\rfloor}^{\lfloor{kn/K}\rfloor-1}
\sum_{i=1}^q \P_{\aaa_i}\(  \frac{
  S_{(nt -m)_+}(f)}{\sqrt{a_n}}>\lambda \) \int_{\X} \frac{\P_x(\tau_{{D}}=m)}{\mu(\tau_{{D}}\le n)}\,
 p_m(x,i) \, \pi(dx).\nn
\end{align}
The second equality uses the fact that $f$ is
supported on $D$ and the strong Markov property.
In the last equality, we merely partition $\{0,\ldots,n-1\}$ into $K$ parts .  

  {Suppose} first that $0<\beta\leq 1$.  Working backwards through an
 argument similar to \eqref{eq:asymptotic    sum}, we obtain
\begin{align}\label{e:K.sum}
&\sum_{k=1}^{K} \sum_{m=\lfloor {(k-1)n/K}\rfloor}^{\lfloor{kn/K}\rfloor-1} \sum_{i=1}^q \P_{\aaa_i}\(  \frac{
  S_{n(t-k/K)_+}(f)}{\sqrt{a_n}}>\lambda \) \nn \int_{\X} \frac{\P_x(\tau_{{D}}=m)}{\mu(\tau_{{D}}\le n)}\, p_m(x,i) \pi(dx) \\[5pt]
&= \mu_{n}\( \frac{\  \hat S_{T_n^{K,t}}(f)}{\sqrt{a_n}}>\lambda\)\,,
\end{align} 
where
$$
T_n^{K,t}(\bx):= \( nt -  \bigl( nk/K-\tau_D(\bx)\bigr) \)_+  \ \ \text{if
  $\tau_{{D}}(\bx) \in \bigl[ (k-1)n/K, kn/K \bigr)$.} 
$$
Clearly, 
$$
\frac{\bigl|  nt -T_n^{K,t}(\bx)\bigr|}{n}\leq 1/K\,.
$$

By Theorem \ref{thm:thm1} and Lemma \ref{l:RV.wander}, we see that
\begin{align}
&\sum_{k=1}^K \sum_{m=\lfloor {(k-1)n/K}\rfloor}^{\lfloor{kn/K}\rfloor-1}
\sum_{i=1}^q  \P_{\aaa_i}\(  \frac{
  S_{ n(t-k/K)_+}(f)}{\sqrt{a_n}}>\lambda \)  \int_{\X} \frac{\P_x(\tau_{{D}}=m)}{\mu(\tau_{{D}}\le n)}\,
p_m(x,i) \,\pi(dx)  \label{e:no.dep.m}  \\[5pt]
&\sim \sum_{k=1}^K \P\( (\Gamma(\beta+1))^{1/2}\sigma_f B(M_\beta(t-k/K)_{+})>\lambda\) \frac{\mu\bigl( \lfloor (k-1)n/K \rfloor \leq \tau_{{D}} < \lfloor kn/K \rfloor -1\bigr)}{\mu(\tau_{{D}}\le n)}  \notag \\[5pt]
&\to \sum_{k=1}^K \Bigl( \bigl(k/K\bigr)^{1-\beta} - \bigl(
  (k-1)/K\bigr)^{1-\beta}\Bigr) \P\( (\Gamma(\beta+1))^{1/2}\sigma_f
  B(M_\beta(t-k/K)_{+})>\lambda\)\,.  \notag 
\end{align}
Combining \eqref{e:no.dep.m} with  \eqref{e:K.sum} implies that
$$
\frac{\  \hat S_{T_n^{K,t}}(f)}{\sqrt{a_n}} 
\Longrightarrow  (\Gamma(\beta+1))^{1/2}\sigma_f B(M_\beta(t-\hat
T_{\infty, K})_{+}), 
$$
where $\hat T_{\infty, K}$ is a discrete random variable independent of
$B$ and $M_\beta$ such that
$$
\P\bigl( \hat T_{\infty, K} =k/K\bigr) =  \bigl(k/K\bigr)^{1-\beta} -
\bigl( (k-1)/K\bigr)^{1-\beta}, \ k=1,\ldots, K\,.
$$
 
We claim that for every $\vep>0$
\begin{equation} \label{e:bill.1}
\lim_{K\to\infty} \limsup_{n\to\infty} \mu_{n}\( \frac{\  \sup_{0\leq
s_1,s_2\leq t, \, |s_1-s_2|\leq 1/K} \bigl| \hat S_{  ns_1 }(f)-  \hat S_{ 
  ns_2 }(f)\bigr|}{\sqrt{a_n}}>\vep \) =0.
\end{equation}
Then, since $\hat T_{\infty, K} \Rightarrow T_\infty^1$ as $K\to\infty$,
once we prove \eqref{e:bill.1}, the claim of the
theorem in the case $L=1$ will follow from Theorem
3.2 in \cite{billingsley:1999}.

To see that \eqref{e:bill.1} is true, repeat the steps in
\eqref{eq:asymptotic sum} and bound the probabilities $p_m(x,i)$
from above by 1. We conclude that  \eqref{e:bill.1}
is bounded from above by
$$
\lim_{K \to \infty} \sum_{i=1}^q \limsup_{n\to\infty} \P_{\aaa_i} \( \frac{\  \sup_{0\leq
s_1,s_2\leq t, \, |s_1-s_2|\leq 1/K} \bigl| S_{  ns_1 }(f)-  S_{ 
  ns_2 }(f)\bigr|}{\sqrt{a_n}}>\vep \)
$$
$$
=\lim_{K \to \infty} q\,  \P \( \sup_{0\leq
s_1,s_2\leq t, \, |s_1-s_2|\leq 1/K}   (\Gamma(\beta+1))^{1/2}\sigma_f
\bigl|B(\M(s_1))-B(\M(s_2))\bigr|>\vep\)\,,
$$
where at the second step we used Theorem \ref{thm:thm1}. Now
\eqref{e:bill.1} follows from the sample continuity of the process $\(
B(\M(t)) , t\ge 0\)$.

This proves the required convergence for $L=1$. For general $L$ we
replace $n$ by $nL$, $t$ by $t/L$ and use the regular variation of
$(a_n)$. Using the already considered case $L=1$ we see that
$$
\mu_{nL}\( \frac{\  \hat S_{ 
    nt }(f)}{\sqrt{a_n}}>\lambda \) \to \P\(
(\Gamma(\beta+1))^{1/2}\sigma_f L^{\beta/2} 
B(M_\beta(t/L-T_\infty^{1})_{+}) >\lambda \)\,.
$$
Since $L T_\infty^{1}\eid T_\infty^{L}$ and the process $B(M_\beta)$
is $\beta/2$-self-similar, the claim of the theorem in the case
$0<\beta\leq 1$ has been established.

In the case $\beta=0$  and $L=1$, we proceed as in
\eqref{eq:asymptotic sum}, but 
stop before breaking the sum into $K$ parts. Consider the case
$\lambda\geq 0$; the case $\lambda<0$ can be handled in a similar
manner.  {Fix $t>0$ and choose $\vep>0$ smaller than $t$. Next, }split the sum over $m$ into two
sums; the first over the range $m\leq  n(t-\vep)$, and the second over the
range $n(t-\vep)<m\leq nt$.  {Denote the first sum
$$\Sigma_{n,1}(\lambda):= \sum_{m \leq n(t-\vep)} \int_{\X} \frac{\P_x(\tau_{{D}}=m)}{\mu(\tau_{{D}}\le n)}
\sum_{i=1}^q p_m(x,i)\P_{\aaa_i}\(  \frac{
  S_{(nt -m)_+}(f)}{\sqrt{a_n}}>\lambda \)  \pi(dx)$$
 and denote the second sum, over the
range $n(t-\vep)<m\leq nt$, by  $\Sigma_{n,2}(\lambda)$.}
 Let $0<\rho<1$. By the slow
variation of the sequence $(a_n)$ there is $n_\rho$ such that for all
$n>n_\rho$ and for all $m\leq  n(t-\vep)$, $a_{nt-m}/a_n \in (1-\rho,
1+\rho)$. By Theorem
\ref{thm:thm1} there is $\hat n_\rho$ such that for all $n>\hat n_\rho$, 
$$
\frac{\P_{\aaa_i}\(  \frac{
  S_{n}(f)}{\sqrt{a_n}}>(1\pm \rho)^{-1/2}\lambda \)}{\P\(
 \sigma_f  
B(E_{\rm st}) >(1\pm \rho)^{-1/2}\lambda \)}\in (1-\rho, 1+\rho)\,,
$$
for each $i=1,\ldots, q$, 
where $E_{\rm st}$ is a standard exponential random variable independent of the
Brownian motion. For notational simplicity, we identify $n_\rho$ and
$\hat n_\rho$. We see that for $n> n_\rho$,
$$
(1-\rho)\frac{\mu(\tau_D\leq n(t-\vep))}{\mu(\tau_D\leq n)}
\P\( \sigma_f  B(E_{\rm st}) >(1+ \rho)^{-1/2}\lambda \) \leq \Sigma_{n,1}(\lambda)
$$
$$
\leq (1+\rho)\frac{\mu(\tau_D\leq n(t-\vep))}{\mu(\tau_D\leq n)}
\P\( \sigma_f  B(E_{\rm st}) >(1-\rho)^{-1/2}\lambda \) \,.
$$
Furthermore,
$$
 {\Sigma_{n,2}(\lambda) \leq \frac{\mu(n(t-\vep) <  \tau_D\leq
  nt)}{\mu(\tau_D\leq n)}\,.}
$$
Letting first $n\to\infty$, then $\vep\to 0$, and, finally, $\rho\to
0$, we conclude, by the continuity of the law of $B(E_{\rm st})$ that
$$
\mu_{n}\( \frac{\  \hat S_{nt}(f)}{\sqrt{a_n}}>\lambda \) \to 
t \P\( \sigma_f  B(E_{\rm st}) > \lambda \) \,,
$$
which is the required limit in the case $\beta=0$ and $L=1$. The
extension to the case of a general $L>0$ is the same as in the case
$0<\beta\leq 1$. 

It remains to prove tightness in the case $0<\beta\leq 1$. We will
prove tightness in $\CC[0,1]$. Since we
are dealing with a sequence of processes starting at zero, it is
enough to show that for any $\vep>0$ there is $\delta>0$ such that for
any $n=1,2,\ldots$,
\begin{equation} \label{e:tight.1}
\mu_n\left( \sup_{0\leq s,t\leq 1, \, |t-s|\leq \delta} \frac{1}{\sqrt{a_n}} 
\bigl|\hat S_{nt}(f)- \hat S_{ns}(f)\bigr|>\vep\right)\leq \vep\,.
\end{equation}
However, by the tightness part of Theorem \ref{thm:thm1}, we can
choose $\delta>0$ such that for every $i=1,\ldots, q$ and
$n=1,2,\ldots$, 
$$
\P_{\aaa_i}\left( \sup_{0\leq s,t\leq 1, \, |t-s|\leq \delta}
  \frac{1}{\sqrt{a_n}}  
\bigl|S_{nt}(f)-  S_{ns}(f)\bigr|>\vep\right)\leq \vep\,.
$$
Therefore, arguing as in \eqref{eq:asymptotic sum}, we obtain 
$$
\mu_n\left( \sup_{0\leq s,t\leq 1, \, |t-s|\leq \delta} \frac{1}{\sqrt{a_n}} 
\bigl|\hat S_{nt}(f)- \hat S_{ns}(f)\bigr|>\vep\right)
$$
$$
=\sum_{m=0}^{n} \int_{\X} \frac{\P_x(\tau_{{D}}=m)}{\mu(\tau_{{D}}\le
  n)} \sum_{i=1}^q p_m(x,i) \P_{\aaa_i}\left( \sup_{0\leq s,t\leq 1, \, |t-s|\leq \delta} \frac{1}{\sqrt{a_n}} 
\bigl| S_{(nt-m)_+}(f)-  S_{(ns-m)_+}(f)\bigr|>\vep\right)
  \pi(dx) 
$$
$$
\leq \sum_{m=0}^{n} \int_{\X} \frac{\P_x(\tau_{{D}}=m)}{\mu(\tau_{{D}}\le
  n)} \sum_{i=1}^q p_m(x,i) \P_{\aaa_i}\left( \sup_{0\leq s,t\leq 1, \, |t-s|\leq \delta} \frac{1}{\sqrt{a_n}} 
\bigl| S_{nt }(f)-  S_{ns}(f)\bigr|>\vep\right)
  \pi(dx) \leq \vep\,,
$$
proving \eqref{e:tight.1}. 

Finally, we prove \eqref{e:miss.lemma}. We have 
\begin{align}  \label{dual expression}
\int_E \(\frac{\hat S_n(f)}{\sqrt{a_n}}\)^2\, d\mu_n
&= \frac{1}{a_n\mu(\tau_{{D}}\le n)}\int_E (\hat S_n(f))^2\, d\mu\nn\\
&=  \frac{1}{a_n\mu(\tau_{{D}}\le n)}\[ n\int_{\X} f^2(x)\,
\pi(dx) +2\sum_{j=1}^{n-1} 
\sum_{k=1}^{n-j} \int_\X f(x) P^{k} f(x)\, \pi(dx)\]\,,
\end{align}
where in the first step we used that $f$ is supported by $D$, and in the second step we used the invariance of  {the} measure $\pi$.
By Lemma \ref{l:RV.wander} and  conditions
\eqref{e:cond.1} and \eqref{e:cond.2},
 the  supremum over $n\ge 1$ of the right side of
\eqref{dual expression} is finite.
\end{proof}

%

\section{A mean-zero functional CLT for heavy-tailed infinitely
  divisible processes} \label{sec:FCLT}

We now define precisely the class of \id\ stochastic processes $\BX=(X_1,X_2,\ldots)$
for which we will prove a functional central limit theorem. Those
processes are given in the form \eqref{e:the.process} of a stochastic
integral. 

Let $( E, \EE)$ be the path space of a Markov chain on
$\X$, as in Section \ref{sec:markov}. Let $f:\, \X\to\bbr$ be a
measurable function satisfying \eqref{e:cond.1} and
\eqref{e:cond.2}. We will assume that $f$ is supported by a finite
union of atoms \eqref{e:union.atoms}. Let $h(\bx) := f(x_1)$, $\bx=(x_1,x_2,\dots) \in
E$ be the extension of the function $f$ to the path space $E$ defined
above. 

Let $M$ be a homogeneous
symmetric \id\ random measure $M$ on $(E,\mathcal{E})$ with control
measure $\mu$ given by \eqref{def:mu}. We will assume that the local
L\'evy measure $\rho$ of $M$ has a regularly varying tail with 
index $-\alpha$, $0 < \alpha < 2$:
\begin{equation} \label{eq:regular var}
\rho(\cdot,\infty) \in RV_{-\alpha} \ \ \text{at infinity.}
\end{equation}
Let  
\begin{align}\label{def:eta}
X_k = \int_{E} h\circ T^k(\bx) \, dM(\bx) = \int_{E} f(x_k) \,
  dM(\bx),  \ \ k=1,2,\dots\,,
\end{align}
where $T$ is the left shift on the path space $E$. Since the function
$f$ is supported by a set of a finite measure, it is
straightforward to check that the integrability condition
\eqref{e:integrability} is satisfied, so \eqref{def:eta} presents a well defined
stationary symmetric \id\ process. Furthermore, we have 
$$
 {\P}(X_1>\lambda)\sim  \int_\X |f(x)|^\alpha\, \pi(dx)\,
\rho(\lambda,\infty), \ \ \lambda\to\infty\,,
$$
see e.g. \cite{rosinski:samorodnitsky:1993}. That is, the heaviness of
the marginal  tail of the process $\BX$ is determined by the exponent
$\alpha$ of regular variation  in \eqref{eq:regular var}. On the other
hand, we will assume that the underlying Markov chain is
$\beta$-regular, $0\leq \beta\leq 1$, and we will see that 
the parameter $\beta$ determines the length of memory in the
process $\BX$. 

The main result of this work is the following theorem. Its statement
uses the tail constant $C_\alpha$ of an $\alpha$-stable random
variable; see \cite{samorodnitsky:taqqu:1994}. We also use the inverse
of the tail of the  local L\'evy measure defined by
$$
\rhoinv(y) := \inf\bigl\{ x\geq 0:\, \rho(x,\infty) \leq
 y\bigr\}, \, y>0\,.
$$ 

\begin{thm}  \label{t:main.fclt}
Let $0<\alpha<2$ and $0 \leq \beta \leq 1$.
Suppose that $(Z_n)$ is a $\beta$-regular Harris
chain on $(\X,\XX)$ with an invariant $\sigma$-finite measure
$\pi$. If $\beta=1$, assume that $a_n=o(n)$. Let $f$ be a measurable 
function supported on a finite union of atoms
$D=\cup_{i=1}^q\aaa_i \in \XX^+_0$.  We assume that $f$ satisfies  
\eqref{e:cond.1} and \eqref{e:cond.2}. If $\beta=1$,
we assume also that $f\in L^{2+\vep}(\pi)$ for some $\vep>0$. If
$\alpha\geq 1$, we also assume that for some $\vep>0$,
$\E_{\alpha_1}|f(Z_{\tau_{\alpha_2}})|^{2+\vep}<\infty$ for any two
atoms, $\alpha_1,\alpha_2$, constituting $D$. 

Let $\BX=(X_1,X_2,\ldots)$ be a stationary symmetric \id\ stochastic
process defined in \eqref{def:eta}, where the local L\'evy measure of
the symmetric homogeneous
infinitely divisible random measure $M$ is assumed to satisfy
\eqref{eq:regular var}.  We assume, furthermore, that 
\begin{equation}  \label{e:lower.tail}
x^{p_0} \rho (x,\infty) \to 0 \ \ \text{as } x \downarrow 0
\end{equation}
for some $p_0 \in (0,2)$. Then the sequence  
\begin{equation*} \label{e:cn}
c_n = C_{\alpha}^{-1/\alpha} a_n^{1/2} \rhoinv \bigl( \mu(\,
\tau_{D} \leq n)^{-1} \bigr), \ \ n=1,2,\ldots\,,
\end{equation*}    
satisfies 
\be\label{cn:rv}
c_n  \in RV_{\beta/2 + (1-\beta)/\alpha}\,.
\ee

Let $0<\beta\leq 1$. Then 
\begin{equation} \label{e:main.st}
\frac{1}{c_n}\sum_{k=1}^{n\cdot} X_k \Rightarrow \bigl(
\Gamma(\beta+1)\bigr)^{1/2}\sigma_f
Y_{\alpha,\beta,2}(\cdot) \quad\text{in }\CC[0,\infty)\,, 
\end{equation} 
where $(Y_{\alpha,\beta,2}(t))$ is the process in \eqref{def:Y process}, with the
usual understanding that the sum in the left hand side is defined  {by
 linear} interpolation. 

If $\beta=0$, then \eqref{e:main.st} holds in the sense of convergence
of finite-dimensional distributions. 
\end{thm}


\begin{proof}
The fact that \eqref{cn:rv} holds follows from the assumption of 
$\beta$-regularity  {and} Lemma \ref{l:RV.wander}, taking into account that 
 the regular variation of $\rho$ at infinity implies $\rhoinv \in
RV_{-1/\alpha}$ at zero. For later use we also record now that
\begin{equation}  \label{e:rho.wandering}
\rho (c_na_n^{-1/2}, \infty) \sim C_{\alpha} \,\mu(\tau_{{D}} \leq
n)^{-1} \ \ \text{as $n\to\infty$,}
\end{equation}
which follows directly from the definition of the inverse and the
regular variation of the tail of $\rho$ in \eqref{eq:regular var}.

We start with proving convergence of the finite-dimensional
distributions. It is enough to show that
$$
\frac{1}{c_n} \sum_{j=1}^J \theta_j \sum_{k=1}^{  nt_j }
X_k \Rightarrow  \bigl( \Gamma(\beta+1)\bigr)^{1/2}\sigma_f 
\sum_{j=1}^J \theta_j Y_{\alpha, \beta, 2}(t_j)
$$
for all $J \geq 1$, $0 \leq t_1 < \dots < t_J$, and  {$\theta_1, \dots\,,
\theta_J \in \bbr$.}  We use an argument similar to that in \cite
{owada:samorodnitsky:2015}.

The standard theory of convergence in law
of infinitely divisible  random variables (e.g., Theorem 15.14 in
\cite{kallenberg:2002}), says that we only have to check the
following: in the notation of Theorem \ref{cor:entrance time} and of
Section \ref{sec:limits}, for every $r>0$,
\begin{align}
&\int_{E} \left( \frac{1}{c_n} \sum_{j=1}^J \theta_j \hat S_{ nt_j
  }(f)  \right)^2 \ \ 
\int\limits_0^{rc_n |\sum \theta_j \hat S_{ nt_j }(f)|^{-1}}
  \hspace{-10pt} v \rho(v,\infty) \, dv\, d\mu \label{e:first.cond} \\[5pt]
&\to \frac{r^{2-\alpha} C_\alpha}{2-\alpha} \, 
\bigl( \Gamma(\beta+1)\bigr)^{\alpha/2}\sigma_f^\alpha
  \int_{[0,\infty)} \int_{\Omega^{\prime}} \left| \sum_{j=1}^J
  \theta_j B\Bigl( M_{\beta}\bigl( (t_j-x)_+, \omega^{\prime} \bigr),
  \omega^{\prime} \Bigr) \right|^{\alpha}
  \P^{\prime}(d\omega^{\prime})\,  \nu_\beta(dx)  \notag
\end{align}
and
\begin{align}
&\int_{E} \rho \biggl( rc_n \biggl| \sum_{j=1}^J \theta_j \hat S_{
  nt_j  } (f)\biggr|^{-1}, \infty \biggr) \, d\mu \label{e:omit}
 \\[5pt]
&\to r^{-\alpha} C_\alpha\, \bigl( \Gamma(\beta+1)\bigr)^{\alpha/2} \sigma_f^\alpha
  \int_{[0,\infty)} \int_{\Omega^{\prime}} \left|
 \sum_{j=1}^J \theta_j B\Bigl( M_{\beta}\bigl( (t_j-x)_+,
\omega^{\prime} \bigr), \omega^{\prime} \Bigr) \right|^{\alpha} \P^{\prime}(d\omega^{\prime})
\, \nu_\beta(dx).  \notag
\end{align}
The proof of \eqref{e:omit} is very similar to that of
\eqref{e:first.cond}, so we only prove \eqref{e:first.cond}.

We keep $r>0$ fixed for the duration of the argument. Fix also an
integer $L$ so that $t_J \leq  L$ and define
$$
\psi(y) := y^{-2} \int_0^{ry} x\rho(x,\infty) dx, \ \ \ y > 0\,,
$$
so that the left-hand side in \eqref{e:first.cond} can be expressed as
$$
\int_{E} \psi \left( \frac{c_n}{|\sum_{j=1}^J \theta_j
    \hat S_{ nt_j  }(f)|} \right) d\mu\,.
$$
By  Theorem \ref{cor:entrance time} and the  Skorohod embedding
theorem, there
exists a probability space $(\Omega^{*}, \mathcal{F}^{*}, \P^*)$ and
random variables $Y, Y_1,  Y_2, \dots$ defined on $(\Omega^{*}, \mathcal{F}^{*}, \P^*)$ such that
\begin{align*}
\P^* \circ Y_n^{-1} &= \mu_{nL} \circ \left( \frac{1}{\sqrt{a_n}}
  \sum_{j=1}^J \theta_j \hat S_{ nt_j }(f) \right)^{-1}, \ \
\ n=1,2,\dots\,,  \\[5pt]
\P^* \circ Y^{-1} &= \P^{\prime} \circ \left(  \bigl(
                    \Gamma(\beta+1)\bigr)^{1/2} \sigma_f
                    \sum_{j=1}^J 
  \theta_j B\bigl( M_{\beta}(t_j-T_{\infty}^{L})_+ \bigr)
\right)^{-1} \,, \\[5pt]
Y_n &\to Y, \ \ \ \P^*\text{-a.s.}
\end{align*}
Then 
$$
\int_{E} \psi \left( \frac{c_n}{|\sum_{j=1}^J \theta_j
    \hat S_{ nt_j }(f)|} \right) d\mu = \int_{\Omega^*}
\mu(\tau_{{D}} \leq nL) \, \psi \left( \frac{c_n}{\sqrt{a_n} |Y_n|} \right)
d\P^*.
$$
First, we will establish convergence of the quantity inside the
integral. By Karamata's theorem  (see e.g. Theorem 0.6 in
\cite{resnick:1987}),
\begin{equation} \label{e:psi.rho}
\psi(y) \sim \frac{r^{2-\alpha}}{2-\alpha} \, \rho(y,\infty) \quad \text{as }
 y \to \infty.
\end{equation} 
Therefore, as $n \to \infty$,
\begin{align*}
\mu(\tau_{{D}} \leq nL) \, \psi \left( \frac{c_n}{\sqrt{a_n} |Y_n|} \right)
&\sim \frac{r^{2-\alpha}}{2-\alpha}\,  \mu(\tau_{{D}} \leq nL) \, \rho
\bigl( c_n a_n^{-1/2} |Y_n|^{-1}, \, \infty  \bigr) \\[5pt]
&\sim \frac{r^{2-\alpha}}{2-\alpha}\, |Y_n|^{\alpha} \, \mu(\tau_{{D}}
  \leq nL) \, 
\rho \bigl( c_n a_n^{-1/2}, \, \infty \bigr), \ \ \ \P^*\text{-a.s.},
\end{align*}
where the last line follows from the uniform convergence of regularly
varying functions of negative index; see e.g. Proposition 0.5 in 
\cite{resnick:1987}.
By \eqref{e:rho.wandering} and the regular variation of the wandering
rate in Lemma \ref{l:RV.wander}, we conclude that 
$$
\mu(\tau_{{D}} \leq nL) \, \psi \left( \frac{c_n}{\sqrt{a_n} |Y_n|}
\right) \to \frac{r^{2-\alpha}}{2-\alpha}\, C_{\alpha}\,  L^{1-\beta}
|Y|^{\alpha}, \ \ \ \P^*\text{-a.s.}
$$

It is straightforward to check that
$$
\int_{\Omega^*} L^{1-\beta} |Y|^{\alpha} d\P^* = \bigl(
                    \Gamma(\beta+1)\bigr)^{\alpha/2}\sigma_f^\alpha
                    \int_{[0,\infty)} \int_{\Omega^{\prime}} \left|
                      \sum_{j=1}^J \theta_j B\Bigl( M_{\beta}\bigl(
                      (t_j-x)_+, \omega^{\prime} \bigr),
                      \omega^{\prime} \Bigr) \right|^{\alpha}
                    \P^{\prime}(d\omega^{\prime}) \nu_\beta(dx), 
$$
so it now remains to show that the convergence discussed so far can be
taken under the integral sign. For this, we will use the 
Pratt lemma (see  Exercise 5.4.2.4 in \cite{resnick:1987}). The lemma requires
us to find a sequence of  measurable functions $G_0, G_1, G_2,\dots$
defined on  $(\Omega^*,\mathcal{F}^*,\P^*)$ such that
\begin{align}
 \mu\bigl(\tau_{{D}} \leq nL\bigr) \, \psi \left( \frac{c_n}{\sqrt{a_n}|Y_n|} \right) &\leq G_n \quad \P^*\text{-a.s.}, \label{e:1stPratt} \\[5pt]
G_n &\to G_0 \quad \P^*\text{-a.s.}, \label{e:2ndPratt} \\[5pt]
\E^* G_n   &\to \E^* G_0  \in [0,\infty) \label{e:3rdPratt}.
\end{align}
Throughout the rest of the proof $C$ is a positive constant which may change from line to line.
Note that by \eqref{e:rho.wandering}, 
$\mu(\tau_{{D}} \leq nL) \, \psi(c_na_n^{-1/2})$ tends to a positive finite constant, therefore
$$
\mu\bigl(\tau_{{D}} \leq nL\bigr) \, \psi \left(
  \frac{c_n}{\sqrt{a_n}|Y_n|} \right) \leq C \, \frac{\psi  \bigl(c_n
  a_n^{-1/2} |Y_n|^{-1}\bigr)}{\psi  \bigl(c_n a_n^{-1/2}\bigr)}.
$$
Since $\psi \in RV_{-\alpha}$ at infinity,  Potter's bounds (see
Proposition 0.8 in \cite{resnick:1987}) allow us to write,  
for $0 < \xi < 2-\alpha$:
$$
\frac{\psi \bigl(c_n a_n^{-1/2} |Y_n|^{-1} \bigr)}{\psi
  \bigl(c_n a_n^{-1/2} \bigr)} \ \one \bigl\{ c_n \geq \sqrt{a_n}
|Y_n| \bigr\} \leq C \, \bigl( |Y_n|^{\alpha-\xi} + |Y_n|^{\alpha+\xi}
\bigr) 
$$
for sufficiently large $n$. Further, by \eqref{e:lower.tail},   $y^2
\psi(y) \to 0$ as $y \downarrow 0$, which gives us 
$$
\psi(y) \leq C \, y^{-2} \ \ \text{for all } y \in [0,1].
$$
Thus,
$$
\frac{\psi \bigl(c_n a_n^{-1/2} |Y_n|^{-1}\bigr)}{\psi \bigl(c_n
  a_n^{-1/2}\bigr)} \ \one \bigl\{ c_n < \sqrt{a_n} |Y_n| \bigr\} \leq
C  a_nc_n^{-2} 
\frac{|Y_n|^2}{\psi(c_n a_n^{-1/2})} \,.
$$
Summarizing, for sufficiently large $n$, 
$$
\mu\bigl(\tau_D \leq nL\bigr) \, \psi \left(
  \frac{c_n}{\sqrt{a_n}|Y_n|} \right) \leq C \left(|Y_n|^{\alpha-\xi}
  + |Y_n|^{\alpha+\xi} +   a_nc_n^{-2}  \frac{|Y_n|^2}{\psi(c_n a_n^{-1/2})} \right).
$$
If we we define $G_n$ to be the right-hand side of the above, then 
 \eqref{e:1stPratt} is automatic. 

Let $G_0 := C \bigl( |Y|^{\alpha-\xi} + |Y|^{\alpha+\xi} \bigr)$. It
follows by the definition of $c_n$ and Lemma \ref{l:RV.wander} that 
$$
c_na_n^{-1/2} \geq C \rhoinv\bigl( a_n/n\bigr) \to \infty \ \ \text{as
  $n\to\infty$}
$$
because $a_n/n\to 0$ (this follows  {from regular} variation considerations
if $\beta<1$, and it is assumed to hold if $\beta=1$.) Since
$y^2\psi(y)\to\infty$ as $y\to\infty$ by \eqref{e:psi.rho} and the
fact that $\alpha<2$, we conclude that 
$$
 a_nc_n^{-2}  \frac{|Y_n|^2}{\psi(c_n a_n^{-1/2})}\to 0
$$
$\P^*\text{-a.s.}$, so that  \eqref{e:2ndPratt} holds.  

To show \eqref{e:3rdPratt}, recall  that by Theorem \ref{cor:entrance
  time}, 
$\sup_{n \geq 1}\E^* |Y_n|^2 < \infty$. This implies uniform
integrability of $(|Y_n|^{\alpha \pm \xi}, \, n \geq 1)$ (with respect
to $\P^*$). Combining these observations, 
\begin{align*}
\E^* G_n &= C \left(\E^*|Y_n|^{\alpha-\xi} + \E^*|Y_n|^{\alpha+\xi} +
           a_nc_n^{-2} \frac{\E^*|Y_n|^2}{\psi(c_n a_n^{-1/2})}
           \right) \\[5pt] 
&\to C \bigl( \E^*|Y|^{\alpha-\xi} + \E^*|Y|^{\alpha+\xi} \bigr)=\E^* G_0\,, \ \ \ n \to \infty\,,
\end{align*}
as required. This completes the proof of convergence in finite-dimensional distributions.

It remains to prove tightness in the case $0<\beta\leq 1$. We start by
decomposing the process
$\BX$ according to the magnitude of the L\'evy jumps. Denote 
\begin{align*}
\rho_1 (\cdot) &:= \rho \bigl( \cdot \cap \{ x: |x| > 1 \} \bigr)\,, \\[5pt]
\rho_2 (\cdot) &:= \rho \bigl( \cdot \cap \{ x: |x| \leq 1 \} \bigr)\,,
\end{align*}
and let $M_i$, $i=1,2$ denote independent homogeneous symmetric infinitely
divisible random measures, with the same control measure $\mu$ as $M$, and
local L\'evy measures $\rho_i$, $i=1,2$. Then 
\begin{align*}
\bigl( X_k, \, k=1,2,\ldots\bigr) &\stackrel{d}{=} \left( \int_{E}
  f(x_k) \, dM_1(\bx) + \int_{E} f(x_k)\,  dM_2(\bx),
  \, k=1,2,\ldots\right)  \\
& :=  \bigl( X_k^{(1)} + X_k^{(2)}, \, k=1,2,\ldots\bigr) 
\end{align*}
in the sense of equality of finite-dimensional distributions. 
Notice that $\BX^{(1)}$ and $\BX^{(2)}$ are independent. Furthermore,
since $f\in L^2(\pi)$, we see that
\begin{equation} \label{e:l2.2}
\E (X_k^{(2)})^2 = \int_\X f^2(x)\, \pi (dx)\,  {\int_{-1}^1 y^2}\,
\rho(dy)<\infty\,.
\end{equation} 

Fix $L>0$. 
We will begin with proving  tightness of the normalized partial sums
of  $X_k^{(2)}$ in the space $\mathcal{C}[0,L]$. 
By Theorem 12.3 of \cite{billingsley:1968}, it suffices to show that
there exist $\gamma>1$, $\rho\geq 0$  and $C>0$ such that
\begin{equation}\label{billingsleythm}
\P \left(     
\left|\sum_{k=1}^{ nt } X_k^{(2)}  - \sum_{k=1}^{  ns
   } X_k^{(2)} \right| > \lambda c_n \right) \leq
\frac{C}{\lambda^{\rho}}(t-s)^{\gamma}  
\end{equation}
for all $0   \leq s \leq t \leq L$, $n \geq 1$ and $\lambda >
0$. 

We dispose of the case $n(t-s) < 1$ first, and, in the sequel, we will
assume that $\mu(\tau_D\leq 1)>0$. If this measure is zero, we will
simply replace 1 by a suitable large constant $\gamma$ and dispose of
the case $n(t-s) < \gamma$ first.  It follows from
\eqref{e:l2.2} that 
$$
\P \left(     
\left|\sum_{k=1}^{ nt } X_k^{(2)}  - \sum_{k=1}^{  ns
   } X_k^{(2)} \right| > \lambda c_n \right)
\leq \P \left( \max\bigl(  {|X_1^{(2)}|,\,  |X_2^{(2)}|}\bigr)>\frac{\lambda c_n}{n(t-s)}\right)
$$
$$
\leq C\lambda^{-2} c_n^{-2} n^2 (t-s)^2\,.
$$
It follows from \eqref{cn:rv} that 
$$
c_n^{-2} n^2 \in RV_{2-2(\beta/2 + (1-\beta)/\alpha)} = RV_{1-(1-\beta)(2/ \alpha-1)}\,.
$$
Suppose first that $0 < \beta < 1$. If $(1-\beta)(2/ \alpha-1)>1$, then $n^2/c_n^2$ is bounded by a positive constant, and we are done. 
In the case of $0<(1-\beta)(2/ \alpha-1) \leq 1$, since $n(t-s) < 1$, there is $0<\delta<(1-\beta)(2/ \alpha-1)$ such that 
$$
c_n^{-2} n^2 (t-s)^2 \leq C(t-s)^{1+(1-\beta)(2/ \alpha-1)  -\delta}\,,
$$
which is what is needed for
 \eqref{billingsleythm}. If $\beta=1$, a similar argument works if one
 uses the stronger integrability assumption on $f$ imposed in the
 theorem. 

Let us assume, therefore, that   $n(t-s) \geq 1$.  
By the L\'evy-It{\^o} decomposition,
$$
\sum_{k=1}^{ nt} {X_k^{(2)}} - \sum_{k=1}^{ ns  }X_k^{(2)} 
  \stackrel{d}{=}  \int_{E} \bigl( \hat S_{  nt }(f)   - \hat S_{
   ns  } (f) \bigr)\, dM_2\
$$
$$
\stackrel{d}{=} \iint\limits_{|y(  \hat S_{  nt }(f)   - \hat S_{
   ns  } (f)) |  \leq \lambda c_n} \hspace{-20pt} y \bigl( \hat S_{  nt }(f)   - \hat S_{
   ns  } (f) \bigr)  \,d\bar{N}_2 + \iint\limits_{|y(  \hat S_{  nt }(f)   - \hat S_{
   ns  } (f)) |  > \lambda c_n} 
\hspace{-20pt} y \bigl( \hat S_{  nt }(f)   - \hat S_{
   ns  } (f) \bigr) \,dN_2\,,
$$
where  $N_2$ is a Poisson random measure on $\mathbb{R} \times E$ with
mean measure $\rho_2 \times \mu$  and $\bar{N}_2 := N_2 - \bigl( \rho_2
\times \mu \bigr)$.  Therefore,
\begin{align*}
& \P \left(     
\left|\sum_{k=1}^{ nt } X_k^{(2)}  - \sum_{k=1}^{  ns
   } X_k^{(2)} \right| > \lambda c_n \right)  \notag
\\
&\leq \P\Bigl(\Big| \ \iint\limits_{|y(  \hat S_{  nt }(f)   - \hat S_{
   ns  } (f)) |  \leq \lambda c_n} 
\hspace{-20pt} y \bigl( \hat S_{  nt }(f)   - \hat S_{
   ns  } (f) \bigr)  \,d\bar{N}_2 \Big| > \lambda c_n  \Bigr) \notag \\
&+
 \P\Bigl(\Big| \ \iint\limits_{|y(  \hat S_{  nt }(f)   - \hat S_{
   ns  } (f)) |  > \lambda c_n} 
\hspace{-20pt} y \bigl( \hat S_{  nt }(f)   - \hat S_{
   ns  } (f) \bigr)  \,dN_2 \Big| > 0\Bigr)\,. \label{eq:LevyIto}
\end{align*}
It follows from \eqref{e:lower.tail} that,
\begin{align*}
&\P\Bigl(\Big| \ \iint\limits_{|y(  \hat S_{  nt }(f)   - \hat S_{
   ns  } (f)) |  \leq \lambda c_n} 
\hspace{-20pt} y \bigl( \hat S_{  nt }(f)   - \hat S_{
   ns  } (f) \bigr)  \,d\bar{N}_2 \Big| > \lambda c_n  \Bigr) \\
& \leq \frac{1}{\lambda^2 c_n^2} \E \left| \ \  
\ \iint\limits_{|y(  \hat S_{  nt }(f)   - \hat S_{
   ns  } (f)) |  \leq \lambda c_n} 
\hspace{-20pt} y \bigl( \hat S_{  nt }(f)   - \hat S_{
   ns  } (f) \bigr)  \,d\bar{N}_2 \right|^2 \\
&= \frac{1}{\lambda^2 c_n^2} \  \iint\limits_{|y(  \hat S_{  nt }(f)   - \hat S_{
   ns  } (f)) |  \leq \lambda c_n} 
\hspace{-20pt} \bigl[ y \bigl( \hat S_{  nt }(f)   - \hat S_{
   ns  } (f) \bigr)\bigr]^2  \,d\rho_2\, d\mu  \\
&\leq 4 \int_{E} \left( \frac{ \hat S_{  nt }(f)   - \hat S_{
   ns  } (f)}{\lambda c_n} \right)^2\left(  \int\limits_0^{\lambda c_n /
| \hat S_{  nt }(f)   - \hat S_{   ns  } (f)|} 
\hspace{-20pt} y
  \rho_2(y,\infty)\, dy\right) d\mu \\ 
&\leq \frac{C}{ \lambda^{p_0}} \frac{1}{c_n^{p_0}} \int_{E} | \hat S_{
  nt }(f)   - \hat S_{   ns  } (f)|^{p_0}\, d\mu \,. 
\end{align*}
Similarly,
\begin{align*}
& \P\Bigl(\Big| \ \iint\limits_{|y(  \hat S_{  nt }(f)   - \hat S_{
   ns  } (f)) |  > \lambda c_n} 
\hspace{-20pt} y \bigl( \hat S_{  nt }(f)   - \hat S_{
   ns  } (f) \bigr)  \,dN_2 \Big| > 0\Bigr) \\
&\leq \P\Bigl (N_2 \bigl( \{ 
|y(  \hat S_{  nt }(f)   - \hat S_{
   ns  } (f)) |  > \lambda c_n\} \bigr)\geq 1\Bigr) \\
&\leq \E N_2  \bigl( \{ 
|y(  \hat S_{  nt }(f)   - \hat S_{
   ns  } (f)) |  > \lambda c_n\} \bigr)
\\
&= 2 \int_{E} \rho_2 \bigl (\lambda c_n 
| \hat S_{  nt }(f)   - \hat S_{   ns  } (f))|^{-1}
,\infty\bigr) \, d\mu \\
&\leq  \frac{C}{ \lambda^{p_0}} \frac{1}{c_n^{p_0}} \int_{E} 
 | \hat S_{  nt }(f)   - \hat S_{   ns  } (f))|^{p_0}\, d\mu \,.
\end{align*}
Elementary manipulations of the linear interpolation of the sums and
\eqref{e:miss.lemma} show that 
$$
 \P \left(     
\left|\sum_{k=1}^{ nt } X_k^{(2)}  - \sum_{k=1}^{  ns
   } X_k^{(2)} \right| > \lambda c_n \right) 
\leq \frac{C}{ \lambda^{p_0}} \frac{1}{c_n^{p_0}} \max_{n(t-s)-1\leq
  m\leq n(t-s)+1} 
\int_{E} |\hat S_m (f)|^{p_0}\, d\mu
$$
$$
\leq \frac{C}{\lambda^{p_0}} \frac{1}{c_n^{p_0}}\, \bigl( a_{
  n(t-s)}\bigr)^{p_0/2}\, \mu(\tau_{{D}} \leq   n(t-s))\,,
$$
where at the last step we have used the assumption $n(t-s)\geq 1$, 
regular variation, and Theorem \ref{cor:entrance time}.

Suppose first that $0<\beta<1$. 
Choose $\epsilon > 0$ so that $2/\alpha - \epsilon - 1 > 0$. Note
that, if \eqref{e:lower.tail} holds for some $p_0 \in (0,2)$, it also holds
for all larger $p_0$. Thus, 
we may assume that $p_0$ is close enough to $2$ so that
$$
(1-\beta) \left( \frac{p_0}{\alpha} - \frac{\epsilon p_0}{2} - 1
\right) > \beta \left( 1 - \frac{p_0}{2} \right). 
$$
Since we are  {assuming} that $\mu(\tau_D\leq 1)>0$, we see that
\begin{align*}
&   \frac{1}{c_n^{p_0}}\, \bigl( a_{
  n(t-s)}\bigr)^{p_0/2}\, \mu(\tau_{{D}} \leq   n(t-s))\\[5pt]
&\leq C\left( \frac{a_n \mu(\tau_{{D}} \leq
  n)^{2/\alpha-\epsilon}}{c_n^2} \right)^{p_0/2}\, 
\left( \frac{a_{ n(t-s)}}{a_n} \right)^{p_0/2} \left( \frac{\mu
  \bigl(\tau_{{D}} \leq  n(t-s) \bigr)}{\mu(\tau_{{D}} \leq n)} \right)^{p_0/\alpha - \epsilon p_0 / 2}  \\[5pt]
&\leq  C\left( \frac{a_{ n(t-s)}}{a_n} \right)^{p_0/2} \left( \frac{\mu
  \bigl(\tau_{{D}} \leq  n(t-s) \bigr)}{\mu(\tau_{{D}} \leq n)} \right)^{p_0/\alpha - \epsilon p_0 / 2}\,.
\end{align*}
The first inequality above uses the choice of $\epsilon$ and $p_0$,
while the second inequality follows from the definition of $c_n$ and
regular variation of $\rhoinv$. 
Next, we choose $0 < \xi < \min \{ \beta, 1-\beta \}$ such that
$$
(1-\beta) \left( \frac{p_0}{\alpha} - \frac{\epsilon p_0}{2} - 1 \right) - \beta \left( 1 - \frac{p_0}{2} \right) - \xi \left( \frac{p_0}{2} + \frac{p_0}{\alpha} - \frac{\epsilon p_0}{2} \right) > 0\,.
$$
By the regular variation of $a_n$ and $\mu(\tau_{{D}} \leq n)$,
$$
 \frac{\mu
  \bigl(\tau_{{D}} \leq  n(t-s) \bigr)}{\mu(\tau_{{D}} \leq n)} 
\leq C\, (t-s) ^{1-\beta-\xi}\,, \ \ \ \ \frac{a_{n(t-s)}}{a_n} \leq C\, (t-s)^{\beta-\xi}\,.
$$
Combining these inequalities together, we have
$$
 \P \left(     
\left|\sum_{k=1}^{ nt } X_k^{(2)}  - \sum_{k=1}^{  ns
   } X_k^{(2)} \right| > \lambda c_n \right) 
\leq \frac{C}{\lambda^{p_0}}\, (t-s)^{\gamma},
$$
where $\gamma = (\beta-\xi)\, p_0/2 + (1-\beta-\xi) ( p_0/\alpha - \epsilon p_0/2)$. 
Due to the constraints in $\epsilon$, $p_0$, and $\xi$, it is easy to check that $\gamma>1$. 
This establishes tightness for the normalized partial sums of
$X_k^{(2)}$ in the case $0<\beta<1$. 

If $\beta=1$, then the assumption $a_n=o(n)$ and a standard
modification of Theorem 12.3 of \cite{billingsley:1968} make the same
argument go through. 

It remains to prove tightness of the normalized partial sums
of  $X_k^{(1)}$ in the space $\mathcal{C}[0,L]$, for a fixed
$L>0$. For notational simplicity we take $L=1$. We
will consider first the case $0<\alpha<1$. Let 
$\rho_1^{\leftarrow}(y) := \inf \bigl\{ x \geq 0: \rho_1 (x,\infty) \leq y  \bigr\}$, $y>0$ be.the inverse of the tail of $\rho_1$. 
We will make use of a certain series representation; see \cite{rosinski:1990b}:
\begin{equation} \label{e:split.x1}
\left( \sum_{k=1}^{ nt}  X_k^{(1)}, \,  0\leq t\leq 1 \right) \eid
\left( \sum_{j=1}^{\infty} \epsilon_j \rho_1^{\leftarrow} \left( \frac{\Gamma_j}{2 \mu (\tau_D \leq n)}\right)\, \hat S_{nt} (f)(V_j^{(n)}), \, 0 \leq t \leq 1  \right) \,,
\end{equation}
where $(\epsilon_j)$ is an i.i.d. sequence of Rademacher variables (taking $+1, -1$ with probability $1/2$), $\Gamma_j$ is the $j$th jump time of a unit rate Poisson process,  and $(V_j^{(n)})$ is a sequence of i.i.d. random variables with common law $\mu_n$. Further, $(\epsilon_j)$, $(\Gamma_j)$, and $(V_j^{(n)})$ are taken to be independent with each other. 

Fix $\xi \in (0, 1/\alpha-1)$ and for $K > 2(1/\alpha + \xi) -1$, we split the right-hand side above as follows. 
\begin{align*}
T_{n,1}^{(K)}(t) &= \sum_{j=1}^{K} \epsilon_j \rho_1^{\leftarrow} \left( \frac{\Gamma_j}{2 \mu (\tau_D \leq n)}\right) \hat S_{nt} (f)(V_j^{(n)})\,, \\
T_{n,2}^{(K)}(t) &= \sum_{j=K+1}^{\infty} \epsilon_j \rho_1^{\leftarrow} \left( \frac{\Gamma_j}{2 \mu (\tau_D \leq n)}\right) \hat S_{nt} (f)(V_j^{(n)})\,.
\end{align*}
We will prove that the sequence $(c_n^{-1} T_{n,1}^{(K)})$ is, for every $K$, tight in $\CC[0,1]$, while 
\begin{equation}  \label{e:sec.term}
\lim_{K \to \infty} \limsup_{n \to \infty} \P \bigl( \sup_{0 \leq t \leq 1} |T_{n,2}^{(K)}(t)| > \epsilon c_n \bigr) = 0\,, \ \ \text{for every } \epsilon >0\,.
\end{equation}
Notice that
$$
c_n^{-1} T_{n,1}^{(K)}(t) = C_{\alpha}^{1/\alpha}  \sum_{j=1}^{K} \epsilon_j \rho_1^{\leftarrow} \left( \frac{\Gamma_j}{2 \mu (\tau_D \leq n)}\right) / \rho^{\leftarrow} \bigl( \mu(\tau_D \leq n)^{-1} \bigr) \frac{\hat S_{nt} (f)(V_j^{(n)})}{\sqrt{a_n}}\,.
$$
Since $\rho_1^{\leftarrow}$ and $\rho^{\leftarrow}$ are both regularly varying at zero with exponent $-1/\alpha$,  
$$
\rho_1^{\leftarrow} \left( \frac{\Gamma_j}{2 \mu (\tau_D \leq n)}\right) / \rho^{\leftarrow} \bigl( \mu(\tau_D \leq n)^{-1} \bigr) \to 2^{1/\alpha} \Gamma_j^{-1/\alpha}\,, \ \ \ n \to \infty\,, \ \ \text{a.s.}.
$$
On the other hand, by Theorem \ref{cor:entrance time}, each $a_n^{-1/2} \hat S_{nt}(f)(V_j^{(n)})$ weakly converges in $\CC[0,1]$, and thus, by independence, $c_n^{-1} T_{n,1}^{(K)}$ turns out to be tight in $\CC[0,1]$. 

Next, we will turn to proving \eqref{e:sec.term}. The probability in \eqref{e:sec.term} can be estimated from above by 
$$
\P \left( C_{\alpha}^{1/\alpha} \sum_{j=K+1}^{\infty} \rho_1^{\leftarrow} \left( \frac{\Gamma_j}{2 \mu (\tau_D \leq n)}\right) / \rho^{\leftarrow} \bigl( \mu(\tau_D \leq n)^{-1} \bigr) \sup_{0 \leq t \leq 1} \bigl| a_n^{-1/2} \hat S_{nt}(f)(V_j^{(n)}) \bigr| > \epsilon\, \right)
$$
Appealing to Potter's bounds and the fact that $\rho_1$ has no mass on $\{ x:|x|\leq 1 \}$, 
$$
\rho_1^{\leftarrow} \left( \frac{\Gamma_j}{2 \mu (\tau_D \leq n)}\right) / \rho^{\leftarrow} \bigl( \mu(\tau_D \leq n)^{-1} \bigr) \leq C \max \bigl\{ \Gamma_j^{-1/\alpha+\xi},  \Gamma_j^{-1/\alpha-\xi} \bigr\}\,.
$$
Combining this bound, Chebyshev's inequality, and the Cauchy-Schwarz inequality, the probability in \eqref{e:sec.term} is bounded from above by 
$$
C \epsilon^{-2} \E (B_n)^2 \E \left( \sum_{j=K+1}^{\infty} \max \bigl\{ \Gamma_j^{-1/\alpha+\xi},  \Gamma_j^{-1/\alpha-\xi} \bigr\}  \right)^2,
$$
where $B_n = \sup_{0 \leq t \leq 1} \bigl| a_n^{-1/2} \hat S_{nt}(f)(V_1^{(n)}) \bigr|$. We know from Remark \ref{rk:l2} that the sequence $(\E (B_n)^2)$ is uniformly bounded in $n$. Because of the restriction in $K$, 
$$
\E \left( \sum_{j=K+1}^{\infty} \Gamma_j^{-(1/\alpha \pm \xi)} \right)^2 \leq \left( \sum_{j=K+1}^{\infty} \bigl\{ \E \Gamma_j^{-2(1/\alpha \pm \xi)}  \bigr\}^{1/2} \right)^2 \leq C \left( \sum_{j=K+1}^{\infty} j^{-(1/\alpha \pm \xi)} \right)^2,
$$
where the rightmost term vanishes as $K \to \infty$, so the proof of the tightness has been completed.

This proves tightness in the case $0<\alpha<1$, and we proceed now to
show tightness of the normalized partial sums
of  $X_k^{(1)}$ in the space $\mathcal{C}[0,1]$,  {for} the case $1\leq \alpha<2$. 
Recall that, in this case, we impose a stronger integrability
assumption on $f$. We start with the L\'evy-It{\^o} decomposition
\eqref{e:split.x1} and write, for  $K>1$,  
 \begin{align*}
\left( \sum_{k=1}^{nt }X_k^{(1)}, \, 0\leq t\leq 1\right)
  & \stackrel{d}{=} \left( \sum_{k=1}^{ nt } \(X_k^{(1,K)}
   +X_k^{(2,K)} \), \, 0 \leq t \leq 1 \right) \\
& := \left( \ \iint\limits_{|y|\le Kc_n/\sqrt{a_n}} \hspace{-20pt} y\hat
  S_{ nt} (f) \,d{N}_1 + \iint\limits_{|y|> Kc_n/\sqrt{a_n}}
  \hspace{-20pt} y \hat S_{ nt} (f) \,dN_1, \, 0 \leq t \leq 1 \right).
\end{align*}
Note that the probability that the process $\bigl(\sum_{k=1}^{nt}X_k^{(2,K)}\bigr)$
does not  {identically} vanish on the interval $[0,1]$ does not exceed  
$$
 \P\bigl(N_1  \{ (\bx,y):\, |y|  >  Kc_n/\sqrt{a_n}, \, \tau_{{D}}(\bx)\le n\} \geq 1 \bigr)
 $$
$$
\leq \E N_1 \{ (\bx,y):\, |y|  >  Kc_n/\sqrt{a_n}, \, \tau_{{D}}(\bx)\le n\}
$$
$$
= 2   \rho_1 \left( K c_n/
    \sqrt{a_n}, \infty \right)  \mu(\tau_{{D}}\le n)
$$
$$
\sim 2C_\alpha\frac{\rho \left( K c_n/
    \sqrt{a_n}, \infty \right)}{\rho \left( c_n/
    \sqrt{a_n}, \infty \right)}\to 2C_\alpha K^{-\alpha}
$$
as $n\to\infty$, and this can be made arbitrarily small by choosing
$K$ large. Therefore, we only need to show  tightness, for every fixed
$K$, of the
normalized partial sums of the process $X_k^{(1,K)}$. As in
\eqref{billingsleythm}, it is enough to prove that there exist
$\gamma>1$, $\rho\geq 0$  and $C>0$ such that 
\be\label{eq:final tightness}
 \P \left(     
\left|\sum_{k=1}^{ nt } X_k^{(1,K)}  - \sum_{k=1}^{  ns
   } X_k^{(1,K)} \right| > \lambda c_n \right) \leq
\frac{C}{\lambda^{\rho}}(t-s)^{\gamma}  
\ee
for all $0   \leq s \leq t \leq 1$, $n \geq 1$ and $\lambda >
0$. In a manner, similar to the one we employed while proving
\eqref{billingsleythm}, we can dispose of the case $n(t-s) < 1$, so we
will look at the case $n(t-s) \geq 1$.

Let $0<\vep<1$ be such that $f\in L^{2+\vep}(\pi)$. By Proposition
\ref{p:symm.mom},  
\bea \label{e:last.split}
&&\frac{1}{c_n^{2+\vep}}\E
\left|\sum_{k=1}^{ nt } X_k^{(1,K)}  - \sum_{k=1}^{  ns
   } X_k^{(1,K)} \right|^{2+\vep} \\
\nn&\le&\frac{C}{ c_n^{2+\vep}} \int_{\R\times E}  \bigl| \hat S_{nt}(f)-\hat S_{ns}(f)\bigr|^{2+\vep}
 {|y|^{2+\vep}   1_{\{ |y|\le Kc_n/\sqrt{a_n}\}} } d\rho_1 d\mu   \\
\nn&+& \frac{C}{ c_n^{2+\vep}} \left( \int_{\R\times E}  \bigl| \hat S_{nt}(f)-\hat S_{ns}(f)\bigr|^{2}
 {|y|^{2}   1_{\{ |y|\le Kc_n/\sqrt{a_n}\}}}  d\rho_1
       d\mu\right)^{1+\vep/2}\,.  
\eea
By Karamata's theorem,
$$
\int_\bbr   {|y|^{2+\vep}   1_{\{ |y|\le Kc_n/\sqrt{a_n}\}}}  d\rho_1
\leq C\frac{\bigl( c_n/\sqrt{a_n}\bigr)^{2+\vep}}{\mu(\tau_{{D}} \leq n)}\,.
$$
Further, by the fact that $f$ is supported on ${D}$, and by the
integrability assumption on $f$, we know that
$$
\int_E \bigl| \hat S_{nt}(f)-\hat S_{ns}(f)\bigr|^{2+\vep}\, d\mu 
\leq C \mu(\tau_{{D}} \leq n(t-s))(a_{n(t-s)}) ^{(2+\vep)/2}\,;
$$
see Remark \ref{rk:higher.mom}. By Lemma \ref{l:RV.wander}, $
\mu(\tau_{{D}} \leq n) a_n^{(2+\vep)/2}$ is regularly varying with exponent
bigger than $1$, so the first term in the right hand side of
\eqref{e:last.split} is bounded from above by $C(t-s)^\gamma$, for
some $\gamma>1$. A similar argument produces the same bound for 
second  term in the right hand side of \eqref{e:last.split}. Now an
appeal to Markov's inequality proves \eqref{eq:final tightness}. 
\end{proof}

\section{Appendix: Fractional moments of \id\ random
  variable} \label{sec:moments} 

In this appendix we present explicit bounds on the fractional moments
of certain \id\ random variables in terms of moments of their L\'evy
measures. These estimates are needed for the proof of Theorem
\ref{t:main.fclt}. We have not been able to find such bounds in the
literature. Combinatorial identities for the integer moments have
been known at least since \cite{bassand:bona:1990}.  
Fractional moments have been investigated, using fractional calculus,
by \cite{matsui:pawlas:2014},  
but the latter paper does not give general explicit bounds of the type
we need.

We will consider fractional moments of nonnegative \id\ random
variables 
and of symmetric \id\ random variables in the ranges needed in the
present paper, but our approach can be extended to moments of all
orders. We start with nonnegative \id\ random variables with Laplace
transform of the form
\begin{equation} \label{e:pos.id}
 {\E} e^{-\theta X} = \exp\left\{ -\int_0^\infty \left( 1- e^{-\theta
      y}\right)\, \nu_+(dy)\right\} := e^{-I(\theta)}, \ \theta\geq 0\,,
\end{equation} 
with the L\'evy measure $\nu_+$ satisfying 
$$
\int_0^\infty y \, \nu_+(dy)<\infty\,.
$$
\begin{proposition} \label{p:pos.mom}
Let $1<p<2$. Then there is $c_p\in (0,\infty)$, depending only on $p$,
such that for any \id\ random variable $X$ satisfying
\eqref{e:pos.id}, 
\begin{equation} \label{e:pos.mom}
 {\E} X^p \leq c_p\left( \int_0^\infty y^p\, \nu_+(dy) + \left(
    \int_0^\infty y\, \nu_+(dy) \right)^p\right)\,.
\end{equation}
\end{proposition}
\begin{proof}
If the $p$th moment of the L\'evy measure,
$$
\int_0^\infty y^p \, \nu_+(dy)\,,
$$
is infinite, then so is the left hand side of \eqref{e:pos.mom}, and
the latter  {trivially} holds. Therefore, we will assume for the  {duration}
of the proof that the $p$th moment of the L\'evy measure is finite.
We reserve the notation $c_p$ for a generic finite positive constant
(that may depend only on $p$), and that may change from line to
line. We start with an elementary observation: there is $c_p$ such
that for any $x>0$, 
$$
x^p = c_p\int_0^\infty \left( 1-e^{-xy}\right)^2 y^{-(p+1)}\, dy\,.
$$
Therefore,
\begin{equation} \label{e:mom.intergral}
 {\E} X^p = c_p\int_0^\infty  {\E}\left( 1-e^{-yX}\right)^2 y^{-(p+1)}\, dy
= c_p\int_0^\infty \left( 1-2e^{-I(y)} + e^{-I(2y)}\right)
y^{-(p+1)}\, dy\,,
\end{equation}
where $I$ is defined in \eqref{e:pos.id}. Denote
$$
\theta^+=\sup\bigl\{ \theta\geq 0:\, I(\theta)\leq 1\bigr\}\in
(0,\infty]\,.
$$
Observe that 
\begin{equation} \label{e:cutoff}
\theta^+\geq \left( \int_0^\infty y\, \nu_+(dy) \right)^{-1}\,.
\end{equation}
To see that, notice that, if $\theta^+<\infty$, then
$$
1= I(\theta^+) \leq \theta^+\int_0^\infty y\, \nu_+(dy)\,.
$$

We now split the integral in \eqref{e:mom.intergral} and write
$$
 {\E} X^p = c_p\int_0^{\theta^+} \cdot + c_p\int_{\theta^+}^\infty \cdot \ :=
A+B\,.
$$
Note that by \eqref{e:cutoff},
$$
B\leq c_p \int_{\theta^+}^\infty y^{-(p+1)}\, dy \leq c_p \left(
    \int_0^\infty y\, \nu_+(dy) \right)^p\,.
$$
Next, using the inequality $1-e^{-2\theta}\leq 2(1-e^{-\theta})$ for
any $\theta\geq 0$, see that $I(\theta)\leq I(2\theta)\leq 2I(\theta)$
for each $\theta\geq 0$. Note also that for $0\leq b\leq 2a$ we have
$$
1- {2e^{-a} + e^{-b} \leq 2a^2} + (2a-b)\,,
$$
and we conclude that
$$
A \leq c_p\int_0^{\theta^+} \left( \int_0^\infty  \left( 1- e^{-x
      y}\right)\, \nu_+(dx)\right)^2\, y^{-(p+1)}\, dy
+ c_p\int_0^{\theta^+}  \left( \int_0^\infty  \left( 1- e^{-x
      y}\right)^2\, \nu_+(dx)\right)\, y^{-(p+1)}\, dy\,.
$$
Using the fact that for $0\leq y\leq \theta^+$ we have
$$
I(y)\leq \min\left( 1,y  \int_0^\infty x\, \nu_+(dx)\right)\,,
$$
 {it follows that}
$$
\int_0^{\theta^+} \left( \int_0^\infty  \left( 1- e^{-x
      y}\right)\, \nu_+(dx)\right)^2\, y^{-(p+1)}\, dy 
\leq \int_0^{\left( \int_0^\infty x\, \nu_+(dx) \right)^{-1}}
y^2 \left( \int_0^\infty x\, \nu_+(dx) \right)^{2}\, y^{-(p+1)}\, dy
$$
$$
+  \int_{\left( \int_0^\infty x\, \nu_+(dx) \right)^{-1}}^\infty \,
y^{-(p+1)}\, dy
 {\le} c_p \left(
    \int_0^\infty y\, \nu_+(dy) \right)^p\,.
$$
Finally,
$$
\int_0^{\theta^+}  \left( \int_0^\infty  \left( 1- e^{-x
      y}\right)^2\, \nu_+(dx)\right)\, y^{-(p+1)}\, dy
$$
$$
\leq \int_0^\infty \left( \int_0^\infty \left( 1- e^{-x
      y}\right)^2\, y^{-(p+1)}\, dy\right)\, \nu_+(dx)
= c_p \int_0^\infty y^p\, \nu_+(dy) \,,
$$
and the proof is complete. 
\end{proof}

We consider next a symmetric infinitely divisible random variable,
with characteristic function of the form
\begin{equation} \label{e:0mean.chf}
 {\E} e^{i\theta Y} = \exp\left\{ \int_{-\infty}^\infty \left( e^{i\theta
      y}-1-i\theta y\right)\, \nu(dy)\right\}, \ \theta\in\bbr\,,
\end{equation}
for some symmetric L\'evy measure $\nu$, satisfying 
$$
\int_{|y|\geq 1} y^2\, \nu(dy)<\infty\,.
$$
\begin{proposition} \label{p:symm.mom}
Let $2<p<4$. Then there is $c_p\in (0,\infty)$, depending only on $p$,
such that for any symmetric \id\ random variable $Y$ satisfying
\eqref{e:0mean.chf},
\begin{equation} \label{e:symm.mom2}
 {\E}|Y|^p \leq c_p\left( \int_{-\infty}^\infty |y|^p\, \nu(dy) + \left(
    \int_{-\infty}^\infty y^2\, \nu(dy) \right)^{p/2}\right)\,.
\end{equation}
\end{proposition}
\begin{proof}
Once again, we may and will assume that the moments of the L\'evy measure in
the right hand side of  {\eqref{e:symm.mom2}} are finite. We start with
the case when $\nu(\bbr)<\infty$. If $(W_j)$ is a sequence of
i.i.d. random variables with the common law $\nu/\nu(\bbr)$,
independent of a Poisson random variable $N$ with mean $\nu(\bbr)$,
then 
$$
Y \eid \sum_{j=1}^N W_j\,,
$$
and so by the Marcinkiewicz-Zygmund inequality (see e.g. (2.2), p. 227
in \cite{gut:2009}), 
$$
 {\E}|Y|^p\leq c_p  {\E}\left( \sum_{j=1}^N W_j^2\right)^{p/2}\,.
$$
The random variable
$$
X= \sum_{j=1}^N W_j^2
$$
is a nonnegative random variable with Laplace transform of the form
\eqref{e:pos.id}, with L\'evy measure $\nu_+$ given by 
$$
\nu_+(A) = \nu\{ y:\, y^2\in A\}, \ A \ \text{Borel.}
$$
Applying Proposition \ref{p:pos.mom} (with $p/2$), proves
 {\eqref{e:symm.mom2}} in the compound Poisson case $\nu(\bbr)<\infty$. In
the general case we use an approximation procedure. For $m=1,2,\ldots$
let $\nu_m$ be the restriction of the L\'evy measure $\nu$ to the set
$\{ y:\, |y|>1/m\}$. Then each $\nu_m$ is a finite symmetric
measure. If $Y_m$ is an \id\ random variable with the characteristic 
function given by \eqref{e:0mean.chf}, with $\nu_m$ replacing  {$\nu$, 
then} $Y_m\Rightarrow Y$ as $m\to\infty$, and the fact that
 {\eqref{e:symm.mom2}}  holds for $Y$ follows from the fact that it holds
for each $Y_m$ and Fatou's lemma. 
\end{proof}

\end{document}